\documentclass[a4paper,11pt]{article}
\usepackage{amsmath, amsfonts, amscd, amssymb, amsthm, enumerate, xypic, float}

\def\bfB{\mathbf{B}}

\newcommand{\Irr}{\operatorname{Irr}}
\newcommand{\Hom}{\operatorname{Hom}}
\newcommand{\Mat}{\operatorname{M}}

\newcommand{\Mats}{\operatorname{S}}

\newcommand{\id}{\operatorname{id}}
\newcommand{\GL}{\operatorname{GL}}
\newcommand{\rl}{\operatorname{rl}}

\newcommand{\Ker}{\operatorname{Ker}}

\newcommand{\Vect}{\operatorname{span}}
\newcommand{\im}{\operatorname{Im}}

\newcommand{\tr}{\operatorname{tr}}
\newcommand{\Tr}{\operatorname{Tr}}
\newcommand{\Sp}{\operatorname{Sp}}
\newcommand{\rk}{\operatorname{rk}}

\newcommand{\Gal}{\operatorname{Gal}}

\newcommand{\botoplus}{\overset{\bot}{\oplus}}
\renewcommand{\setminus}{\smallsetminus}

\def\F{\mathbb{F}}
\def\K{\mathbb{K}}
\def\R{\mathbb{R}}
\def\C{\mathbb{C}}

\def\N{\mathbb{N}}

\def\H{\mathbb{H}}

\renewcommand{\L}{\mathbb{L}}

\def\calC{\mathcal{C}}

\def\calL{\mathcal{L}}

\def\calU{\mathcal{U}}

\def\lcro{\mathopen{[\![}}
\def\rcro{\mathclose{]\!]}}

\theoremstyle{definition}
\newtheorem{Def}{Definition}[section]

\theoremstyle{plain}
\newtheorem{theo}{Theorem}[section]
\newtheorem{prop}[theo]{Proposition}

\newtheorem{lemma}[theo]{Lemma}

\theoremstyle{plain}

\theoremstyle{remark}
\newtheorem{Rems}{Remarks}[section]
\newtheorem{Rem}[Rems]{Remark}

\title{Products of involutions in symplectic groups over general fields (II)}
\author{Cl\'ement de Seguins Pazzis\footnote{Universit\'e de Versailles Saint-Quentin-en-Yvelines, Laboratoire de Math\'ematiques
de Versailles, 45 avenue des Etats-Unis, 78035 Versailles cedex, France}
\footnote{e-mail address: clement.de-seguins-pazzis@ac-versailles.fr}}

\begin{document}

\thispagestyle{plain}

\maketitle

\begin{abstract}
Let $s$ be an $n$-dimensional symplectic form over a field $\F$ of characteristic other than $2$, with $n>2$.

In a previous article, we have proved that if $\F$ is infinite then every element of the symplectic group $\Sp(s)$ is the product of four involutions if $n$ is a multiple of $4$ and of five involutions otherwise.

Here, we adapt this result to all finite fields with characteristic not $2$, with the sole exception of the very special situation where $n=4$ and $|\F|=3$,
a special case which we study extensively.
\end{abstract}




\section{Introduction}

\subsection{The problem}

Throughout, $\N$ will denote the set of all non-negative integers (following French notation), and $\N^*$ the set of all positive ones.

An element of a group $G$ is called \textbf{$k$-reflectional} when it is the product of $k$ involutions of $G$
(i.e. elements $x$ of $G$ such that $x^2=1_G$).

Let $(V,s)$ be a symplectic space of dimension $n$ over an arbitrary field $\F$ with characteristic other than $2$
(i.e.\ $V$ is a vector space of dimension $n$ over $\F$, and $s$ is a symplectic form on $V$, that is a non-degenerate alternating bilinear form on $V$).
Recently, there has been renewed interest in the structure of the symplectic group $\Sp(s)$
with respect to the involutions and more precisely in the problem of decomposing an element of $\Sp(s)$ into a product of involutions with minimal number of factors.
Since there are only trivial involutions in $\Sp(s)$ (that is $\pm \id_V$) if $n=2$, we will assume $n \geq 4$ from now on.

Strikingly, good results have appeared only recently in the literature.
De La Cruz \cite{dLC} proved that if the underlying field is algebraically closed, then every element of $\Sp(s)$
is $4$-reflectional. It was proved in \cite{dSPinvol4Symp} (section 7.1 there) that this result is entirely optimal, for all
values of $n \geq 4$. Awa and de La Cruz \cite{Awa} later solved the case $n=4$ for the field of real numbers, and almost simultaneously
Ellers and Villa \cite{EllersVilla} generalized de La Cruz's result for all $n$ that are multiples of $4$ and all fields in which the polynomial $t^2+1$ has a root:
Unfortunately, their method is nothing short of a miracle and has no reasonable extension to other fields.

In the recent \cite{dSPinvol4Symp}, a breakthrough was made. Thanks to a new method called ``space-pullback", it was proved
that over any infinite field:
\begin{itemize}
\item every element of $\Sp(s)$ is $4$-reflectional if $n$ is a multiple of $4$;
\item every element of $\Sp(s)$ is $5$-reflectional if $n$ is not a multiple of $4$.
\end{itemize}
Whether the second result is optimal remains an open problem, whereas the adaptation of the first result to finite fields has been shown to fail
for $\F_3$ and $n=4$ (section 7.3 in \cite{dSPinvol4Symp}).
Interestingly, the algebraic properties of the underlying field played no part in the methods of \cite{dSPinvol4Symp}:
the key feature was the infiniteness of the field, which allowed one to make a careful choice of parameters.

The present article is devoted to the adaptation of the previous results to finite fields (still with characteristic other than $2$;
in sharp contrast, for fields with characteristic $2$ it is known that every element of a symplectic group is $2$-reflectional \cite{Gow}).
It turns out that new methods are required to handle such fields, and in particular the case of
$\F_3$ is very difficult. As said earlier it was proved that some elements of
$\Sp_4(\F_3)$ are not $4$-reflectional. Fortunately, this case is the only exception to the previous results, but
this failure profoundly complicates the strategy of proof for $\F_3$.

Fortunately, when dealing with symplectic forms over finite fields, what we lose in the variety of scalars is gained in the simplicity of the structure of the conjugacy classes in $\Sp(s)$. Indeed, in $\Sp(s)$ the conjugacy class of an element $u$ is not only encoded in the invariant factors of $u$ seen an a vector space endomorphism, but also in additional invariants which we call the \emph{Wall\footnote{We could as well have attributed these invariants to
Williamson and Springer \cite{Springer}.} invariants}:
\begin{itemize}
\item For each irreducible monic palindromial (see our precise definition in Section \ref{section:basicpol}) $p\in \F[t] \setminus \{t \pm 1\}$ and each integer $r \geq 1$, there is a Wall invariant attached to $u$ with respect to $(p,r)$ in the form of the isometry class of a nondegenerate \emph{skew-Hermitian form} over the residue field $\F[t]/(p)$ (equipped with the non-identity involution that takes the coset of $t$ to its inverse), and the rank of this skew-Hermitian form is the number of primary invariants of $u$ that equal $p^r$, i.e.\ the Jordan number of $u$ attached to the pair $(p,r)$;
\item For each $\eta=\pm 1$ and each even integer $r \geq 2$, there is a Wall invariant attached to $u$ with respect to $(t-\eta,r)$ in the form of the isometry class of a nondegenerate \emph{symmetric bilinear form} over $\F$ whose rank is the number of Jordan cells of size $r$ for $u$ with respect to the eigenvalue $\eta$.
\end{itemize}
Since here the field $\F$ is systematically taken finite, the skew-Hermitian Wall invariants are redundant with the plain Jordan numbers, as their isometry class is already encoded in their rank, which is simply a Jordan number of $u$. This is a consequence of the classical fact that nondegenerate Hermitian forms
over finite fields are classified by their rank. In contrast, Wall invariants with respect to pairs of the form $(t\pm 1,r)$ cannot be overlooked, but at least their complexity is limited by the classification of quadratic forms over $\F$, which is not difficult.

In contrast with \cite{dSPinvol4Symp}, in which we entirely avoided  Wall invariants, for finite fields we will need to care about them, but fortunately
they are far more simple than for infinite fields!

The main technique used in \cite{dSPinvol4Symp} is what we call the \emph{space-pullback technique}.
Here, apart from this technique, which will remain our main tool, we will use a collection of additional techniques:
\begin{itemize}
\item Algebraic model techniques, in which specific pairs $(s,u)$ are represented by using field extensions of $\F$;
\item Cyclic matrices fitting techniques, reminiscent of de La Cruz's approach from \cite{dLC};
\item And also purely geometric methods, where we look at the action of $u \in \Sp(s)$ on totally $s$-singular subspaces (these techniques are used
to deal with the case where $n$ is not a multiple of $4$).
\end{itemize}

\subsection{Main result, and structure of the article}

Our main result is the following theorem.

\begin{theo}[Main theorem]\label{theo:main}
Let $(V,s)$ be a symplectic space of dimension $n>2$ over a finite field $\F$.
\begin{enumerate}[(a)]
\item If $n =0 \mod 4$, then every element of $\Sp(s)$ is $4$-reflectional, unless $\dim V=4$ and $|\F|=3$.
\item If $n = 2 \mod 4$, then every element of $\Sp(s)$ is $5$-reflectional
\item If $n=4$ and $|\F|=3$, then every element of $\Sp(s)$ is $5$-reflectional.
\end{enumerate}
\end{theo}

Result (c) is optimal, as shown in section 7.3 of \cite{dSPinvol4Symp}. In Section \ref{section:F3}, we will give a complete picture of $\Sp_4(\F_3)$ with respect to $k$-reflectionality,
i.e.\ for each conjugacy class in $\Sp_4(\F_3)$ we will find the least number of involutions that are necessary to decompose its elements into products of involutions.

In \cite{dSPinvol4Symp}, the main way to prove the equivalent of Theorem \ref{theo:main} for infinite fields was by induction on the dimension (by steps of four).
Yet such a strategy appears to be infeasible for finite fields, and much more profound and precise methods are needed. Unfortunately, at the present point of the article
it is not possible to give any sensible idea of our methods, and we have to review much material before we can finally sketch the new ideas.
So, let us immediately describe the overall structure of the article.

\begin{itemize}
\item The first four sections that follow the present introduction consist of reviews of known essential material on symplectic groups and decompositions into products of involutions. First of all, in Section \ref{section:conjugacyclasses}, we review the classification of conjugacy classes in symplectic groups, as
    well as the classification of conjugacy classes of indecomposable elements. Then, in Section \ref{section:review},
    we recall some results from \cite{dSPinvol4Symp} that give wide ranges of specific elements of $\Sp(s)$ that admit decompositions into a small number of involutions, as well as results that can easily be deduced from known results on such decompositions in general linear groups. In particular, we review Nielsen's theorem, in which the $2$-reflectional elements of $\Sp(s)$ are characterized, and we give two key examples of $3$-reflectional elements.
    In Section \ref{section:specificcells}, we obtain decompositions of specific cells thanks to techniques in general linear groups, as well
    as ideas from Galois theory.
    Finally, in Section \ref{section:spacepullback}, we review and expand on the ``space-pullback" method that was used in \cite{dSPinvol4Symp}
    to obtain the equivalent of Theorem \ref{theo:main} for infinite fields. Here, the new results deal with the way of obtaining very large subspaces on which
    the space-pullback method can be applied (Proposition \ref{prop:existpresquelagrangian}), and also the way of adapting the method to obtain very specific results (Lemmas \ref{lemma:adaptlemmairr} and \ref{lemma:adaptlemmasplit}).

\item After these preliminary technical parts, we will be ready to prove Theorem \ref{theo:main}. First of all, we will prove point (a) in
Section \ref{section:dimmultiple4}. The study is cut into two parts: given $u \in \Sp(s)$, different methods are required according as the total number of Jordan cells of odd size of $u$ for an eigenvalue in $\{\pm 1\}$ is a multiple of $4$ or not.

\item Point (c) will then be proved in Section \ref{section:F3}, in which we will actually give a full account of the structure of $\Sp_4(\F_3)$:
for each conjugacy class $\calC$ we will give the minimal length for a decomposition of its elements into a product of involutions.

\item The last section (Section \ref{section:dim2mod4}) is devoted to the proof of point (b) of Theorem \ref{theo:main}.
\end{itemize}

In a large part of the study, substantial shortcuts could be obtained by discarding fields with $3$ elements, but we have chosen to
keep the treatment as general as possible throughout.

\section{A review of conjugacy classes in symplectic groups}\label{section:conjugacyclasses}

\subsection{Basic considerations on polynomials}\label{section:basicpol}

Throughout, $t$ is an indeterminate for polynomials. We denote by $\F[t]$ the algebra of all polynomials in the indeterminate $t$ and
with coefficients in $\F$, and by $\Irr(\F)$ the set of all monic irreducible polynomials of $\F[t]$ that are distinct from $t$
(i.e.\ that have no zero root).
For any $p \in \F[t]$ of degree $d$ such that $p(0) \neq 0$, one defines the \textbf{reciprocal polynomial}
$$p^\sharp:=p(0)^{-1} t^d p(t^{-1})$$
of $p$, which is monic of degree $d$. One sees that $(p^\sharp)^\sharp=p$ whenever $p$ is monic.
Note that $p^\sharp$ is irreducible whenever $p$ is irreducible. One says that $p$ is a \textbf{palindromial} whenever $p=p^\sharp$.
Classically, every $p \in \Irr(\F) \setminus \{t\pm 1\}$ which is a palindromial has even degree and satisfies $p(0)=1$.

\subsection{Cyclic endomorphisms}\label{section:cyclic}

Remember that an endomorphism $u$ of a finite dimensional vector space $V$ is called \textbf{cyclic} whenever there exists a vector
$x \in V$ such that $\{u^k(x) \mid k \in \N\}$ spans $V$. In that case, $u$ is determined up to conjugation in the general linear group of $V$
by its minimal polynomial, which coincides with its characteristic polynomial. This justifies the following terminology:

\begin{Def}
Let $p \in \F[t]$ be a monic polynomial with degree $n$. A \textbf{$\calC(p)$-endomorphism} of a finite-dimensional vector space $V$
is a cyclic endomorphism of $V$ with minimal (i.e. characteristic) polynomial $p$.
\end{Def}

To speak more quickly, we will simply that an endomorphism is $\calC(p)$ to mean that it is a $\calC(p)$-endomorphism.

\subsection{The viewpoint of pairs}

\begin{Def}
An \textbf{s-pair} $(s,u)$ consists of a symplectic form $s$ on a finite-dimensional vector space $V$ over $\F$, and of an $s$-symplectic transformation $u$
(i.e.\ $\forall (x,y)\in V^2, \; s(u(x),u(y))=s(x,y)$). We say that $V$ is the underlying vector space of $(s,u)$, and that
its dimension is the \textbf{dimension} of $(s,u)$. An s-pair is called \textbf{trivial} when its dimension is zero.
\end{Def}

Two s-pairs $(s,u)$ and $(s',u')$, with underlying vector spaces $V$ and $V'$, are called \textbf{isometric} whenever there exists
a linear isometry $\varphi$ from $(V,s)$ to $(V',s')$ (that is, $\varphi : V \overset{\simeq}{\rightarrow} V'$ is a linear bijection and
$\forall (x,y)\in V^2, \; s'(\varphi(x),\varphi(y))=s(x,y)$) such that $u'=\varphi \circ u \circ \varphi^{-1}$.
This defines an equivalence relation on the collection of all s-pairs.

Let $(s,u)$ be an s-pair, with underlying vector space $V$. Assume that we have a splitting $V=V_1 \overset{\bot_s}{\oplus} V_2$ into $s$-orthogonal subspaces,
both stable under $u$. Then $V_1$ and $V_2$ are $s$-regular, and we denote by $s_1$ and $s_2$ the resulting symplectic forms on $V_1$ and $V_2$, yielding
s-pairs $(s_1,u_{V_1})$ and $(s_2,u_{V_2})$. We shall write $u=u_{V_1} \botoplus u_{V_2}$ without any reference to $s$
(in practice, the symplectic form under consideration is always obvious from the context). We say that $(s,u)$ is the (internal) \textbf{orthogonal sum}
of $(s_1,u_{V_1})$ and $(s_2,u_{V_2})$. The generalization to an internal orthogonal sum of more than two s-pairs is effortless.
We say that  $(s,u)$ is \textbf{indecomposable} when it is non-trivial and there is no decomposition $V=V_1 \botoplus V_2$ into non-trivial $s$-orthogonal subspaces that are stable under $u$.

Obviously, every s-pair splits as an orthogonal direct sum of finitely many indecomposable s-pairs (and potentially zero such pairs).

\subsection{Symplectic extensions}

Let $v$ be an automorphism of a finite-dimensional vector space $V$. Denote by $V^\star:=\Hom(V,\F)$ the
dual vector space of $V$, and by $v^t : \varphi \in V^\star \mapsto \varphi \circ v \in V^\star$ the transpose of $v$.
For a linear subspace $W$ of $V$, we denote by $W^\circ:=\{f \in V^\star : \; f_{|W}=0\}$ its dual-orthogonal in $V^\star$.

The (external) symplectic extension of $v$ is the s-pair $S(v)$ defined as follows with
underlying vector space $V \times V^\star$: the symplectic form is
$$((x,\varphi),(y,\psi)):=\varphi(y)-\psi(x)$$
and the automorphism is
$$u : (x,\varphi) \longmapsto \bigl(v(x),(v^{-1})^t(\varphi)\bigr).$$
In this article, we will need the internal viewpoint for these pairs.
So, let $(s,u)$ be an s-pair with underlying vector space $V$ and dimension $n$.
Remember that a \textbf{Lagrangian} of $s$ is a linear subspace $\calL$ of $V$ that is totally $s$-singular (that is
$\forall (x,y) \in V^2, \; s(x,y)=0$) and that has dimension $\frac{n}{2}\cdot$
Two Lagrangians $\calL$ and $\calL'$ are called transverse whenever $\calL \cap \calL'=\{0\}$, i.e.\ $V=\calL \oplus \calL'$.
Now, assume that we have two such Lagrangians that are stable under $u$; then $u$ is entirely determined by their data and by
the automorphism $v:=u_\calL$ of $\calL$ that $u$ induces: more specifically, the symplectic form $s$
induces a vector space isomorphism $\varphi : x \in \calL' \mapsto s(x,-) \in \calL^\star$, and one proves that
$u_{\calL'}=\varphi^{-1} \circ (u_\calL)^t \circ \varphi$ by using the fact that $u$ is $s$-symplectic.

Conversely, let $\calL$ and $\calL'$ be transverse Lagrangians of $(V,s)$, and let $v$ be a vector space automorphism of $\calL$.
Then, the unique automorphism of $V$ whose restriction to $\calL$ is $v$ and whose restriction to $\calL'$ is
$\varphi^{-1} \circ v^t \circ \varphi$ is $s$-symplectic: we call it the \textbf{symplectic extension of $v$ to $\calL'$}
and denote it by $s_{\calL'}(v)$.

In matrix terms, if we have a symplectic basis $\bfB$ of $V$ that is adapted to $V=\calL \oplus \calL'$, then
the symplectic extensions to $\calL'$ of the automorphisms of $\calL$ are the automorphisms that are represented in
$\bfB$ by the matrices of the form $A \oplus A^\sharp$ with $A \in \GL_d(\F)$ (for $d:=\dim \calL$), where for an arbitrary invertible matrix $M \in \GL_n(\F)$
we set
$$M^\sharp:=(M^T)^{-1}.$$

\begin{Rem}\label{remark:pseudoextension}
Let $\calL$ be a Lagrangian of $V$ and $\bfB$ be a symplectic basis whose first $n$ vectors span $\calL$.
Then every $u \in \Sp(s)$ that leaves $\calL$ invariant has its matrix in $\bfB$ of the form
$$\begin{bmatrix}
A & ? \\
0 & A^\sharp
\end{bmatrix},$$
and in particular if $u_{\calL}=\varepsilon \id$ for some $\varepsilon=\pm 1$ then
the above matrix takes the form
$$\begin{bmatrix}
\varepsilon I_n & ? \\
0 & \varepsilon I_n
\end{bmatrix}.$$
\end{Rem}

\begin{lemma}\label{lemma:splitsymplecticextension}
Let $(V,s)$ be a symplectic space and $\calL$ be a Lagrangian of it.
Assume that some automorphism $v$ of $\calL$ splits into $v=v_1 \oplus v_2$, and denote respectively by
$\calL_1$ and $\calL_2$ the corresponding linear subspaces of $\calL$.

Then every symplectic extension of $v$ is symplectically similar to the orthogonal direct sum of symplectic extensions of
$v_1$ and $v_2$.
\end{lemma}

\begin{proof}
Let $\calL'$ be a Lagrangian that is transverse to $\calL$.
Let us consider the isomorphism $\varphi : x \in \calL' \overset{\simeq}{\longmapsto} s(x,-) \in \calL^\star$.

Setting $\calL'_1:=\varphi^{-1}(\calL_2^{\circ})$ and $\calL'_2:=\varphi^{-1}(\calL_1^{\circ})$, we
see that $V=(\calL_1 \oplus \calL'_1) \botoplus_s (\calL_2 \oplus \calL'_2)$, and in particular
both $\calL_1 \oplus \calL'_1$ and $\calL_2 \oplus \calL'_2$ are $s$-regular.
Moreover, all of $\calL_1,\calL'_1,\calL_2,\calL'_2$ are stable under $s_{\calL'}(u)$; as
for all $i \in \{1,2\}$ the resulting endomorphism of $\calL_i \oplus \calL'_i$ is $s$-symplectic, it must equal $s_{\calL'_i}(v_i)$.
Hence $s_{\calL'}(u)=s_{\calL'_1}(v_1) \botoplus s_{\calL'_2}(v_2)$, which yields the claimed result.
\end{proof}

\subsection{Recognizing symplectic transformations}

The following result was proved in \cite{dSPinvol4Symp} (combine lemma 2.2 and corollary 2.7 there):

\begin{lemma}\label{lemma:recog}
Let $(s,u)$ be an s-pair, with underlying vector space $V$. Assume that we have a totally $s$-singular subspace $W$ of $V$
that is stable under $u$. Denote by $p$ the characteristic polynomial of $u_{W}$, by
$\overline{s}$ the symplectic form on $W^{\bot}/W$ induced by $s$, and by
$\overline{u}$ the $\overline{s}$-symplectic transformation of $W^{\bot}/W$ induced by $u$.
Assume finally that the characteristic polynomial of $\overline{u}$ is relatively prime with $pp^\sharp$.
Then there is a splitting $(s,u)=(s_0,u_0) \botoplus (s_1,u_1)$ in which:
\begin{enumerate}[(i)]
\item The underlying vector space $V_0$ of $(s_0,u_0)$ includes $W$ as a Lagrangian that is stable under $u_0$ with induced endomorphism equal to $u_W$.
\item The s-pair $(s_1,u_1)$ is isometric to $(\overline{s},\overline{u})$.
\end{enumerate}
If in addition $p$ is relatively prime with $p^\sharp$ then $u_0$ is a symplectic extension of $u_W$.
\end{lemma}

\subsection{A classification of indecomposable s-pairs}

Using a part of Wall's classification of conjugacy classes in symplectic groups \cite{Wall}, one can
retrieve the structure of indecomposable s-pairs.
Here, it will be convenient to sort these s-pairs into six types, and for the sake of clarity we
draw a table of these six types.

\begin{table}[H]
\caption{The classification of indecomposable s-pairs $(s,u)$}
\label{figure1}
\begin{center}
\begin{tabular}{| c || c | c |}
\hline
Type & Type of $u$ as an endomorphism & Associated data  \\
\hline
\hline
I & $\calC(p^{2n})$-automorphism & $p \in \Irr(\F) \setminus \{t\pm 1\}$ palindromial \\
  &                                           &  $n \in \N^*$ \\
\hline
II & $\calC(p^{2n+1})$-automorphism & $p \in \Irr(\F) \setminus \{t\pm 1\}$ palindromial \\
  &                                           &  $n \in \N$ \\
  \hline
III & symplectic extension of a   & $q \in \Irr(\F)$ with $q \neq q^\sharp$ \\
  & $\calC(q^{2n})$-automorphism    &  $n \in \N^*$ \\
\hline
IV & symplectic extension of a   & $q \in \Irr(\F)$ with $q \neq q^\sharp$  \\
  &  $\calC(q^{2n+1})$-automorphism   &  $n \in \N$ \\
\hline
V & $\calC((t-\eta)^{2n})$-automorphism  & $\eta=\pm 1$, $n \in \N^*$  \\
\hline
VI & symplectic extension of a & $\eta =\pm 1$, $n \in \N$ \\
   & $\calC((t-\eta)^{2n+1})$-automorphism &   \\
   \hline
\end{tabular}
 \end{center}
\end{table}

Note that the indecomposable s-pairs of type VI are the only ones with a non-cyclic component $u$.

In \cite{dSPinvol4Symp}, the terminology was slightly different as it was not needed there to differentiate between the indecomposable cells of type I and II
nor between the indecomposable cells of type III and IV; here the distinction will be relevant for dimensions that are not multiples of $4$.

\subsection{Wall invariants}

In the present work, we will not need the full power of Wall's classification of conjugacy classes in symplectic groups, but it is necessary
to state some partial results.

To every s-pair $(s,u)$ we attach three kinds of invariants:
\begin{itemize}
\item The \textbf{Jordan numbers} of $u$: for each $p \in \Irr(\F)$ and each integer $r \geq 1$,
we denote by $n_{p,r}(u)$ the number of primary invariants of $u$ that equal $p^r$;
this number equals the dimension, over the residue field $\F[t]/(p)$,
of the cokernel $V_{p,r}$ of the mapping $\Ker p(u)^{r+1}/\Ker p(u)^{r} \hookrightarrow \Ker p(u)^{r}/\Ker p(u)^{r-1}$
induced by $p(u)$. If $p \neq p^\sharp$ then it turns out that $n_{p,r}(u)=n_{p^\sharp,r}(u)$
(because $u$, being $s$-symplectic, is similar to $u^{-1}$ in the general linear group of its underlying vector space).

\item The \textbf{quadratic Wall invariants} of $u$: For each $\eta \in \{\pm 1\}$ and each even integer $r=2k+2 \geq 2$,
the quadratic Wall invariant $(s,u)_{t-\eta,r}$ is the nondegenerate symmetric bilinear form induced by
$$(x,y) \mapsto \frac{1}{2}\,s\bigl(x,(u-u^{-1})(u+u^{-1}-2\eta\id)^k(y)\bigr)$$
on the quotient vector space $V_{t-\eta,r}$.

\item The \textbf{Hermitian Wall invariants} of $u$: For each \emph{palindromial} $p$ in $\Irr(\F) \setminus \{t\pm 1\}$ and each integer $r \geq 1$,
we attach to $(s,u)$ a nondegenerate skew-Hermitian form $(s,u)_{p,r}$ over $\L:=\F[t]/(p)$ (for the non-identity involution that takes the class of $t$ to its inverse) on the $\L$-vector space $V_{p,r}$. We will not give the precise construction of this skew-Hermitian form and simply
refer the interested reader to section 1.4 of \cite{dSP2invol}.
\end{itemize}

\begin{theo}[Wall's theorem]\label{theo:Wall}
Let $\F$ be a field of characteristic other than $2$.
Two s-pairs $(s,u)$ and $(s',u')$ over $\F$ are isometric if and only if all the following three conditions are fulfilled:
\begin{enumerate}[(i)]
\item The endomorphisms $u$ and $u'$ are similar (as endomorphisms of vector spaces over $\F$), or equivalently
$n_{p,r}(u)=n_{p,r}(u')$ for all $p \in \Irr(\F)$ and all $r \in \N^*$.
\item For every $\eta =\pm 1$ and every even integer $r\geq 2$, the quadratic Wall invariants $(s,u)_{t-\eta,r}$ and $(s',u')_{t-\eta,r}$
are equivalent (as symmetric bilinear forms over $\F$).
\item For every palindromial $p \in \Irr(\F) \setminus \{t \pm 1\}$ and every integer $r \geq 1$, the (nondegenerate)
skew-Hermitian forms $(s,u)_{p,r}$ and $(s',u')_{p,r}$ are equivalent over the field $\L:=\F[t]/(p)$.
\end{enumerate}
\end{theo}

Over a finite field equipped with a non-identity involution, nondegenerate Hermitian forms are classified by their rank,
and hence so are skew-Hermitian forms (they are in natural correspondence with Hermitian forms, by multiplying Hermitian forms with a fixed
non-zero skew-Hermitian element of the field). And here the rank of $(s,u)_{p,r}$ is precisely the dimension of $V_{p,r}$
as a vector space over $\F[t]/(p)$, i.e.\ it is the Jordan number $n_{p,r}(u)$. Hence, over finite fields the Hermitian Wall invariants are irrelevant
for the classification of conjugacy classes, and only the quadratic Wall invariants need to be considered in addition to the Jordan numbers:

\begin{theo}[Wall's theorem for finite fields]
Let $(s,u)$ and $(s',u')$ be two s-pairs over a finite field $\F$ of characteristic other than $2$. For $(s,u)$ and $(s',u')$ to be isometric, it is then necessary and sufficient that both the following conditions hold:
\begin{enumerate}[(i)]
\item One has $n_{p,r}(u)=n_{p,r}(u')$ for all $p \in \Irr(\F)$ and all $r \geq 1$.
\item For every $\eta =\pm 1$ and every even $r \geq 2$, the quadratic Wall invariants $(s,u)_{t-\eta,r}$ and $(s',u')_{t-\eta,r}$
are equivalent as symmetric bilinear forms over~$\F$.
\end{enumerate}
\end{theo}

In the present work, we will try avoiding situations where the second condition is needed to identify an isometry of s-pairs.

\section{A review of known results on decompositions into involutions}\label{section:review}

\subsection{Basic considerations}

We start with very basic yet crucial remarks.

\begin{Rems}
\begin{enumerate}[(i)]
\item Let $u$ be a $k$-reflectional element of a symplectic group. Then $u$ is also $l$-reflectional in the said symplectic group for all $l \geq k$.

\item Let $(s,u)$ and $(s',u')$ be isometric s-pairs. If $u$ is $k$-reflectional in $\Sp(s)$ then $u'$ is
$k$-reflectional in $\Sp(s')$.

\item Let $(s,u)$ be an s-pair, and $k$ be a positive integer.
Assume that we have a splitting $(s,u) \simeq (s_1,u_1) \botoplus (s_2,u_2)$ in which
$u_i$ is $k$-reflectional in $\Sp(s_i)$ for all $i \in \{1,2\}$. Then $u$ is $k$-reflectional in $\Sp(s)$.
\end{enumerate}
\end{Rems}

The following terminology will also be quite useful.

\begin{Def}
Let $u_1$ and $u_2$ be two elements of a symplectic group $\Sp(s)$.
We say that $u_1$ is \textbf{i-adjacent} to $u_2$ whenever there exists an involution $i \in \Sp(s)$
such that $u_1=i u_2$.
\end{Def}

\begin{Rem}
If $u_1$ is i-adjacent to a $k$-reflectional element $u_2$, then $u_1$ is $(k+1)$-reflectional.
\end{Rem}

\begin{Rem}
Assume that we have a decomposition $u=u_1 \botoplus u_2$, and $u_1$ and $u_2$ are respectively i-adjacent to
transformations $u'_1$ and $u'_2$ (in the corresponding symplectic groups). Then
$u$ is i-adjacent to $u'_1 \botoplus u'_2$.
\end{Rem}

\subsection{Symplectic extensions}

Let $(V,s)$ be a symplectic space and $(\calL,\calL')$ be a pair of transverse Lagrangians of it.
The mapping
$$s_{\calL'} : v \in \GL(\calL) \mapsto s_{\calL'}(v) \in \Sp(s)$$
is easily seen to be a group homomorphism. This yields the following result, which was also featured in \cite{dSP2invol}:

\begin{prop}\label{prop:sympextensionrefl}
Let $(s,u)$ be an s-pair, and let $k \in \N^*$. If $(s,u)$ is a symplectic extension of an automorphism
that is $k$-reflectional (in the corresponding general linear group), then $u$ is $k$-reflectional in $\Sp(s)$.
\end{prop}

\subsection{The characterization of $2$-reflectional elements}\label{section:Nielsen}

It is critical now that we state Nielsen's characterization of products of two involutions in symplectic groups.
This result is proved in the recent \cite{dSP2invol} (Nielsen never published the result, although it was stated in B\"unger's PhD thesis
\cite{Bungerthesis}). Note that the implication (ii) $\Rightarrow$ (i) is a special case of Proposition \ref{prop:sympextensionrefl}.

\begin{theo}[Nielsen's theorem]
Let $(s,u)$ be an s-pair. The following conditions are equivalent:
\begin{enumerate}[(i)]
\item $u$ is the product of two involutions in $\Sp(s)$;
\item $u$ is a symplectic extension of an automorphism that is the product of two involutions (in the corresponding general linear group);
\item $u$ is a symplectic extension of an automorphism that is similar to its inverse;
\item $u$ is a symplectic extension of an automorphism whose invariant factors are palindromials;
\item All the Jordan numbers of $u$ are even, and all the Wall invariants of $(s,u)$ are hyperbolic.
\end{enumerate}
\end{theo}

For finite fields, the hyperbolicity of a Hermitian Wall invariant is equivalent to having its underlying vector space of even dimension.
Hence, for such fields condition (v) is limited to the statement that all the Jordan numbers of $u$ are even and that all the
\emph{quadratic} Wall invariants of $(s,u)$ are hyperbolic. In particular, if $u$ has no eigenvalue in $\{\pm 1\}$, or if it has only Jordan cells
of odd size for those eigenvalues, then it is $2$-reflectional in $\Sp(s)$ if and only if all its Jordan numbers are even.

Now, here are some useful examples of s-pairs that satisfy the conditions in this theorem:

\begin{itemize}
\item If $u$ is an indecomposable cell of type VI, then it is $2$-reflectional (e.g.\ by characterization (iv)).
\item If $\F$ is finite then for every s-pair $(s,u)$ such that $u^2=-\id$ and $\rk u$ is a multiple of $4$, the transformation
$u$ is $2$-reflectional in $\Sp(s)$. Indeed, in that case either $t^2+1$ is irreducible over $\F$, in which case
$u$ has a sole non-zero Jordan number, namely $n_{t^2+1,1}(u)$, and it equals $\frac{\rk u}{2}$;
or $t^2+1$ splits over $\F$, with roots $\pm \imath$, in which case $u$ has exactly two non-zero Jordan numbers, both equal to
$\frac{\rk u}{2}\cdot$
Note that this uses the finiteness of $\F$ in a critical way, as is demonstrated by the case of the field $\R$ of real numbers:
in that case indeed, Hermitian forms over $\R[t]/(t^2+1) \simeq \C$ are not classified solely by their rank, so the corresponding Hermitian Wall invariants
need a close examination.
\item In contrast (and for the very same reasons), if we have an s-pair $(s,u)$ such that $u^2=-\id$ and $\rk u$ is not a multiple of $4$, then $u$ is not $2$-reflectional in $\Sp(s)$.
\end{itemize}

\subsection{Cyclic adaptation lemma}

Let us start from a classical lemma on cyclic automorphisms of a vector space:

\begin{lemma}[Proposition 3.4 in \cite{dSP3invol}]\label{lemma:cyclicfitplain}
Let $q \in \F[t]$ be a monic polynomial of degree $d>0$ such that $q(0) \neq 0$.
Let $u$ be a $\calC(q)$-automorphism of a vector space $V$, and
let $r \in \F[t]$ be monic of degree $d$.
If either $r(0)=-q(0)$, or $r(0)=q(0)$ with $d$ odd, then there exists an involution $i \in \GL(V)$ such that $iu$ is $\C(r)$.
\end{lemma}

From there, we readily deduce:

\begin{lemma}\label{lemma:cyclicfit}
Let $(V,s)$ be a symplectic space, and $\calL$ and $\calL'$ be two transverse Lagrangians of it.
Let $u$ be a $\calC(q)$-automorphism of $\calL$, and let
$r \in \F[t]$ be a monic polynomial with the same degree as $q$.

If either $r(0)=-q(0)$, or $r(0)=q(0)$ with $d$ odd, then $s_{\calL'}(u)$ is i-adjacent to $s_{\calL'}(v)$
for some $\calC(r)$-automorphism $v$ of $\calL$.
\end{lemma}

\subsection{Some useful $3$-reflectional elements}

We will need two important results which allow one to recognize $3$-reflectional elements.
The first one is easily obtained by combining Proposition \ref{prop:sympextensionrefl} with the classical result
that states that every cyclic automorphism with determinant $\pm 1$ is $3$-reflectional in the corresponding general linear group
(see e.g.\ proposition 3.7 of \cite{dSP3invol}):

\begin{prop}\label{prop:sympextension3refl}
Let $(s,u)$ be an s-pair. Assume that $u$ is a symplectic extension of a cyclic automorphism whose minimal polynomial $p$ satisfies $p(0)=\pm 1$.
Then $u$ is $3$-reflectional in $\Sp(s)$.
\end{prop}

The second result is more difficult. It was proved in \cite{dSPinvol4Symp} as an application of the space-pullback technique.
It will also play a crucial part in the present work.

\begin{prop}[Proposition 3.6 in \cite{dSPinvol4Symp}]\label{prop:even3refl}
Let $(s,u)$ be an s-pair. Assume that $u$ is a symplectic extension of a cyclic automorphism $v$ whose minimal polynomial
is even and relatively prime with its reciprocal polynomial. Then $u$ is $3$-reflectional in $\Sp(s)$.
\end{prop}

Note that in Proposition \ref{prop:even3refl}, $v$ need not be a product of involutions in the corresponding general linear group (as we might have $\det v \neq \pm 1$!).

\section{On specific indecomposable cells}\label{section:specificcells}

Here, we will illustrate the use of previous results by tackling some types of indecomposable cells. The results here will be used only for the case of $\F_3$
in dimension $4$, but the proofs are interesting in themselves.

\subsection{A lemma on polynomials}

We start with a lemma on polynomials. Remember that $\chi(\F) \neq 2$.

\begin{lemma}\label{lemma:polynomp(0)-1}
Let $n \geq 2$ be an integer. Assume that $\F$ is finite.
Then there exists a monic irreducible polynomial $p\in \F[t]$ of degree $n$ such that $p(0)=-1$.
\end{lemma}

\begin{proof}
We choose a finite extension $\L$ of degree $n$ of $\F$, and a generator $\zeta$ of the group $\L^\times$ of units.
Denote by $q$ the cardinality of $\F$ and set $k:=\frac{q-1}{2}$ if $n$ is even, $k:=q-1$ if $n$ is odd, and $z:=\zeta^k$;
in any case, we denote by $p$ the minimal polynomial of $z$ over $\F$.
Since $k q^l<q^n-1$ for all $l \in \lcro 0,n-1\rcro$, the elements
$z,z^q,\dots,z^{q^{n-1}}$, which are the images of $z$ under the elements of $\Gal(\L/\F)$, are pairwise distinct, and hence $p$ has degree $n$.
For the constant coefficient, we compute
$$p(0)=(-1)^{n}\prod_{k=0}^{n-1} z^{q^k}=(-1)^n z^{\frac{q^{n}-1}{q-1}}.$$
If $n$ is even we have $z^{\frac{q^{n}-1}{q-1}}=\zeta^{\frac{q^n-1}{2}}=-1$. If $n$ is odd
we have $z^{\frac{q^{n}-1}{q-1}}=\zeta^{q^n-1}=1$. In any case we deduce that $p(0)=-1$.
\end{proof}

\subsection{Tackling cells of type I, III or V}

\begin{lemma}\label{lemma:typeI-III-V}
Assume that $\F$ is finite.
Let $(s,u)$ be an s-pair with dimension a multiple of $4$. Assume that it is an indecomposable cell of type I, III or V.
Then $u$ is $4$-reflectional in $\Sp(s)$.
\end{lemma}

\begin{proof}
There is a palindromial $p$ of even degree (not necessarily irreducible) such that $p(0)=1$,
and an integer $n \geq 1$ such that $u$ is $\calC(p^{2n})$.
Set $\calL:=\Ker p(u)^n=\im p(u)^n$, and note that $\calL^{\bot_s}=(\im p(u)^n)^{\bot_s}=\Ker p(u)^n=\calL$.
Hence $\calL$ is a Lagrangian. Let us take a transverse Lagrangian $\calL'$ of $\calL$.
The endomorphism $v$ of $\calL$ induced by $u$ is $\calC(p^n)$.

By Lemma \ref{lemma:polynomp(0)-1}, we can find an irreducible polynomial $q$ of degree $n \deg(p)$ such that $q(0)=-1$.
By the cyclic adaptation lemma (Lemma \ref{lemma:cyclicfit}), we can choose an involution $i$ of the vector space $\calL$ such that $iv$ is
$\calC(q)$. Hence $s_{\calL'}(i)u$ stabilizes $\calL$ and the induced endomorphism is $iv$.
Note that $q$ is irreducible with even degree and $q(0)=-1$, so it is relatively prime with $q^\sharp$.
By Lemma \ref{lemma:recog}, we deduce that $s_{\calL'}(i)u$ is a symplectic extension of $iv$.
Hence $s_{\calL'}(i)u$ is $3$-reflectional by Proposition \ref{prop:even3refl}. Hence, $u$ is $4$-reflectional in $\Sp(s)$.
\end{proof}

\begin{lemma}\label{lemma:from(t2+1)2tot^4+1}
Assume that $|\F|=3$.
Let $(s,u)$ be an s-pair in which $u$ is $\calC\left((t^2+1)^2\right)$.
Then $u$ is i-adjacent to a $\calC(t^4+1)$-symplectic transformation.
\end{lemma}

\begin{proof}
This is readily deduced from the proof of Lemma \ref{lemma:typeI-III-V} by taking $q:=t^2-t-1$, which is irreducible over $\F$ with degree $2$.
With the notation from the proof, the characteristic polynomial of $s_{\calL'}(i)u$ is $qq^\sharp=t^4+1$, which has only simple irreducible factors,
and hence $t^4+1$ is also the minimal polynomial of $s_{\calL'}(i)u$, whereas $s_{\calL'}(i)$ is an involution in $\Sp(s)$.
\end{proof}

\subsection{Tackling cells of type II}

\begin{lemma}\label{lemma:typeII}
Assume that $\F$ is finite.
Let $(s,u)$ be an indecomposable s-pair of type II, with dimension a multiple of $4$.
Then $u$ is $4$-reflectional in $\Sp(s)$.
\end{lemma}

The proof is partly inspired by Ellers and Villa's argument from \cite{EllersVilla}.

\begin{proof}
We write the minimal polynomial of $u$ as $p^{2n+1}$ where $p \in \Irr(\F)$ is a palindromial of even degree $d$. Note that $d$ is a multiple of $4$
because so is $(2n+1)d$.

We start with the case $n=0$.

Consider the splitting field $\L:=\F[t]/(p)$ equipped with the involution $x \mapsto x^\bullet$
that takes the coset $\lambda$ of $t$ to its inverse, and with the norm $N : x \in \L \mapsto xx^\bullet$.

Choose an arbitrary nonzero skew-Hermitian (i.e.\ skew-selfadjoint) element $h$ of $\L$ (say, $h=\lambda-\lambda^{-1}$),
and consider the symplectic form
$$s' : (x,y) \mapsto \Tr_{\L/\F} (h x^\bullet y)$$
on the $\F$-vector space $\L^2$.
Since $\lambda^\bullet \lambda=1$, one sees that $u' : x \mapsto \lambda x$
is $s'$-symplectic. Since $u'$ is $\calC(p)$,
we deduce that $(s,u)$ is isometric to $(s',u')$ (once more, there is no relevant Wall invariant here because $p$
has no root in $\{-1,1\}$ and the field $\F$ is finite). It will suffice to prove that $u'$ is $4$-reflectional in $\Sp(s')$.

Now, denote by $\K$ the subfield of all Hermitian (i.e.\ selfadjoint) elements of $\L$.
Since $d$ is a multiple of $4$, the degree of $\K$ over $\F$ is even, and it follows (because $\K$ is finite) that
there is an element $\imath \in \K$ such that $\imath^2=-1$. Note that $N(\imath)=\imath^2=-1$.
Consider then $r_1 : x \mapsto \imath x^\bullet$ and $r_2 : x \mapsto -\lambda \imath x^\bullet$, which are $\F$-linear (and even $\K$-linear).
Using $N(-\imath\lambda)=-1=N(\imath)$ and the obvious fact that $x \mapsto x^\bullet$ is $s'$-skewsymplectic, we gather that both
$r_1$ and $r_2$ are $s'$-symplectic. Clearly $u'=r_2 r_1$, and $r_1^2=-\id=r_2^2$.

To conclude, it suffices to prove that each one of $r_1$ and $r_2$ is $2$-reflectional in $\Sp(s')$.
But this follows immediately from one of the examples from Section \ref{section:Nielsen}, because $d$ is a multiple of $4$.
Hence, $u'$ is $4$-reflectional in $\Sp(s')$, and we conclude that $u$ is $4$-reflectional in $\Sp(s)$.

Now, we turn to the case $n>0$.
Then $\Ker p^n(u)=\im p^{n+1}(u)=(\Ker p^{n+1}(u))^{\bot_s}$. In particular $\Ker p^n(u)$ is totally $s$-singular.
Let us take a basis $(e_1,\dots,e_N)$ of $\Ker p^n(u)$ (with $N:=nd$), and then extend it to a basis $(e_1,\dots,e_N,f_1,\dots,f_d)$ of
$\Ker p^{n+1}(u)$ such that $(f_1,\dots,f_d)$ is $s$-symplectic. Note that $\Vect(f_1,\dots,f_d)$ is $s$-regular.
Finally, we extend $(e_1,\dots,e_N)$ into a symplectic basis $(e_1,\dots,e_N,g_1,\dots,g_N)$ of $\Vect(f_1,\dots,f_d)^{\bot_s}$.
Hence $\bfB:=(e_1,\dots,e_N,f_1,\dots,f_d,g_1,\dots,g_N)$ is a basis of $V$ and the matrix of $u$ in that basis looks as follows:
$$M=\begin{bmatrix}
A & ? & ? \\
0 & B & ? \\
0 & 0 & A^\sharp
\end{bmatrix}$$
where $A \in \GL_{nd}(\F)$ is cyclic with minimal polynomial $p^n$, and $B\in \Sp_d(\F)$ is cyclic with minimal polynomial $p$.
By the first case we have studied in the above, we find a symplectic involution $S \in \Sp_d(\F)$ such that $SB$ is $3$-reflectional in
$\Sp_d(\F)$. Once more, we use Lemma \ref{lemma:polynomp(0)-1} to find some $q \in \Irr(\F)$ of degree $dn$ such that $q(0)=-1$.
By Lemma \ref{lemma:cyclicfitplain}, we can find an involutory matrix $S' \in \GL_{nd}(\F)$ such that $S'A$ is cyclic with minimal polynomial $q$.
Denote by $i_1$ the endomorphism of $V$ represented in $\bfB$ by
$$\widetilde{S}:=\begin{bmatrix}
S' & 0 & 0 \\
0 & S & 0 \\
0 & 0 & (S')^\sharp
\end{bmatrix}.$$
With the way $\bfB$ has been chosen, it is clear that $i_1$ is $s$-symplectic, and obviously $i_1$ is an involution.
Then, $i_1s$ is represented in $\bfB$ by
$$\begin{bmatrix}
S'A & ? & ? \\
0 & SB & ? \\
0 & 0 & (S'A)^\sharp
\end{bmatrix}.$$
Now,  the characteristic polynomial of $SB$ must be relatively prime with $q$ because its degree is less than or equal to $nd$, whereas $q$
is irreducible and is not a palindromial. Moreover, $q$ is relatively prime with $q^\sharp$.
Hence, by Lemma \ref{lemma:recog} we gather that $i_1u$ splits into the orthogonal direct sum $u_1 \botoplus u_2$, where
$u_1$ is a symplectic extension of a $\calC(q)$-automorphism, and
$u_2$ is $3$-reflectional in the corresponding symplectic group. Finally, by Proposition \ref{prop:sympextension3refl}, $u_1$
is $3$-reflectional in the corresponding symplectic group. Hence $i_1 u$ is $3$-reflectional in $\Sp(s)$, and we conclude that $u$
is $4$-reflectional in $\Sp(s)$.
\end{proof}

\section{The space-pullback technique, recalled and refined}\label{section:spacepullback}

\subsection{A review of the space-pullback technique}

We will recall this technique, which was the key to the main theorem of \cite{dSPinvol4Symp}
(see sections 3.1 and 3.2 there for more details).
The idea is to look at the action on totally $s$-singular subspaces.
So, assume that we have an s-pair $(s,u)$, with underlying vector space $V$,
and a totally $s$-singular subspace $W$ such that $u(W) \cap W=\{0\}$.
Then, we look at the restricted bilinear form
$$s_{u,W} : (x,y) \in W^2 \longmapsto s(x,u(y)),$$
and we assume that it is \emph{nondegenerate.}
Note that its respective symmetric and skew-symmetric parts are
$$(x,y) \in W^2 \longmapsto \frac{1}{2}\, s\bigl(x,(u-u^{-1})(y)\bigr) \quad \text{and} \quad
(x,y) \in W^2 \longmapsto \frac{1}{2}\, s\bigl(x,(u+u^{-1})(y)\bigr).$$
Under the condition of nondegeneracy, the subspace $W \oplus u(W)$ is $s$-regular.
Then we take an arbitrary symplectic form $b$ on $W$.
It turns out that:
\begin{itemize}
\item[(i)] There is a unique symplectic involution $i$ of $W \oplus u(W)$
such that $i(W)=u(W)$ and
$$\forall (x,y)\in W^2, \; b(x,y)=s(x,i(y)).$$
\item[(ii)] There is also a unique endomorphism $v$ of $W$ such that
$$\forall (x,y)\in W^2, \; s_{u,W}(x,y)=b(x,v(y)).$$
\end{itemize}
And then $i(u(x))=v(x)$ for all $x \in W$. This justifies the term ``space-pullback", as $i$ pulls $u(W)$ back into $W$.
Now, we extend $i$ to a symplectic involution $\widetilde{i}$ of $V$ by choosing a symplectic involution of
$(W \oplus u(W))^{\bot_s}$, which will be called the \textbf{residual involution} induced by $\widetilde{i}$, and denoted by $i'$.

The composite symplectic transformation $\widetilde{i} \circ u$ leaves $W$ invariant, and it induces a symplectic transformation of
$W^{\bot_s}/W$. The analysis of this transformation is crucial. Denoting by $\pi$ the canonical projection of $u(W) \oplus (W \oplus u(W))^{\bot_s}$ onto $(W \oplus u(W))^{\bot_s}$, we call the mapping $x \mapsto \pi(u(x))$
(which is well defined on $(W \oplus u(W))^{\bot_s}$) the \textbf{residual endomorphism} of $(W \oplus u(W))^{\bot_s}$ induced by $u$.
If we denote this endomorphism by $u'$, then the symplectic transformation of $W^{\bot_s}/W$ induced by $iu$
is symplectically similar to the symplectic transformation $i'u'$ of $(W \oplus u(W))^{\bot_s}$, and more precisely
the canonical mapping $\varphi : (W \oplus u(W))^{\bot_s} \overset{\simeq}{\longrightarrow} W^{\bot_s}/W$ is such that
$\varphi \circ (i'u') \circ \varphi^{-1}$ is the symplectic transformation of $W^{\bot_s}/W$ induced by $iu$.
Hence, the type of the former is entirely determined by the effect of $\widetilde{i}$ on $(W \oplus u(W))^{\bot_s}$
and by the bilinear form induced by $s_u$ on $(W \oplus u(W))^{\bot_s}$.

In order to make the technique fruitful, we will need careful choices of $W$, $b$ and of a symplectic involution of
$(W \oplus u(W))^{\bot_s}$ (to be the residual involution of the involution $i \in \Sp(s)$ that is constructed in the space-pullback).
In practice, most of the time we do not define $b$ directly, rather we choose a vector space automorphism $v$ of $W$ such that
$(x,y) \mapsto s_{u,W}(x,v^{-1}(y))$ is symplectic, and we try to force the conjugacy class of $v$ in $\GL(W)$.
Note that if $s_{u,W}$ is symmetric, asking that $(x,y) \mapsto s_{u,W}(x,v^{-1}(y))$ be symplectic amounts to asking that
$v$ be $s_{u,W}$-skew-symmetric. And in all but one case (more precisely Lemma \ref{lemma:pairF3t^2+1}), we will try to have $s_{u,W}$ symmetric.

In Section \ref{section:adaptspacepullback}, we will obtain interesting results on the choice of the conjugacy class of $v$. And in Section \ref{section:lagrangianchoice}, we will
discuss how to choose $W$ wisely.

\subsection{Adapting the symplectic form for space-pullbacking}\label{section:adaptspacepullback}

Since a complete understanding of the situation is not required, we will limit ourselves to a very specific yet critical situation.

\begin{lemma}\label{lemma:adaptlemmairr}
Assume that $\F$ is finite. Let $p \in \Irr(\F)$ be even.
Then there exists a pair $(b,v)$ consisting of:
\begin{enumerate}[(i)]
\item A nonhyperbolic nondegenerate symmetric bilinear form $b$ on a vector space $V$;
\item A $b$-skew-selfadjoint $\calC(p)$-endomorphism $v$ of $V$.
\end{enumerate}
\end{lemma}

\begin{proof}
Denote by $d$ the degree of $p$.

We consider the field $\L:=\F[t]/(p)$ equipped with the non-identity involution $x \mapsto x^\bullet$
that takes the coset $\lambda$ of $t$ to its opposite (it is well defined because $p$ is even),
and the $\F$-linear mapping $v : x \mapsto \lambda x$.
We also consider the subfield $\K:=\{x \in \L : \; x^\bullet=x\}$, of which $\L$ is a quadratic extension.
The bilinear form $b: (x,y) \mapsto \Tr_{\L/\F}(x^\bullet y)$
is nondegenerate and symmetric, and $v$ is $b$-skew-selfadjoint. It remains to investigate the type of $b$.

Let us write $d:=\deg p$ and $\Gal(\L/\F)=\{\sigma_1,\dots,\sigma_d\}$. Since $\L$ is a Galois extension of $\F$ (remember that $\L$ is finite) and
$\L=\F[\lambda]$, it is classical that the Gram matrix of $b$ in the basis $(1,\lambda,\dots,\lambda^{d-1})$
equals $(M^\bullet)^T M$, where $M=(\sigma_i(\lambda^{j-1}))_{1 \leq i,j \leq d}$, and hence
one of the determinants of $b$ is $(\det M)^\bullet (\det M)$.

Since $\L$ is finite with even degree over $\F$, there is a unique quadratic extension $\mathbb{M}$ of $\F$
such that $\mathbb{M} \subseteq \L$. As the discriminant of $p$ over $\F$ equals
$\pm (\det M)^2$, we deduce that $\det M \in \mathbb{M}$.
When applying an element of $\Gal(\L/\F)$ entry-wise, the rows of $M$ are permuted, and in particular
$(\det M)^\bullet=\pm \det M$.
The group $\Gal(\L/\F)$ is cyclic and we take a generator $\varphi$ of it.
The morphism $\varphi$ acts as a $d$-cycle on the roots of $p$, and its $\frac{d}{2}$-th power $x \mapsto x^\bullet$ acts as the product of $\frac{d}{2}$-commuting transpositions on them;
hence the signature of the latter permutation equals $(-1)^{d/2}$, so that $(\det M)^\bullet=(-1)^{d/2} \det M$.
The same line of reasoning shows that $\det M \not\in \F$ because $\varphi(\det M)=\det M^{\varphi}=(-1)^{d-1} \det M \neq \det M$.
Hence, for some $\beta \in \mathbb{M} \setminus \F$, the scalar $(-1)^{d/2} \beta^2$ is a determinant of $b$ .

If $b$ were hyperbolic, then $(-1)^{d/2} \beta^2=(-1)^{d/2}\alpha^2$ for some $\alpha \in \F$, leading to $\beta=\pm \alpha$
and contradicting the fact that $\beta \not\in \F$. Hence $b$ is nonhyperbolic.
\end{proof}

\begin{lemma}\label{lemma:adaptlemmasplit}
Let $\lambda \in \F \setminus \{0,1,-1\}$ and $n \geq 1$.
Then there exists a pair $(b,v)$ consisting of:
\begin{enumerate}[(i)]
\item A hyperbolic bilinear form $b$ on a vector space $V$ of dimension $2n$;
\item A $b$-skew-selfadjoint endomorphism $v$ of $V$ that is diagonalizable with eigenvalues $\pm \lambda$, both
with multiplicity $n$.
\end{enumerate}
\end{lemma}

\begin{proof}
It suffices to consider the matrices
$$A:=\begin{bmatrix}
\lambda I_n & 0_n \\
0_n & -\lambda I_n
\end{bmatrix} \quad \text{and} \quad
S:=\begin{bmatrix}
0_n & I_n \\
I_n & 0_n
\end{bmatrix}$$
of $\Mat_{2n}(\F)$. We note that $S$ represents a hyperbolic bilinear form, that $A$ is diagonalizable with eigenvalues $\pm \lambda$, both
with multiplicity $n$, and that $SA$ is alternating. Hence, it suffices to take for $b$ the bilinear form $(X,Y) \mapsto X^T SY$ on $\F^{2n}$, and for
$v$ the endomorphism $X \mapsto AX$ of $\F^{2n}$.
\end{proof}

\subsection{Finding large subspaces for space-pullbacking}\label{section:lagrangianchoice}

Here, we inquire about the existence of large totally $s$-singular subspaces $W$ for which $s_{u,W}$ is symmetric and nondegenerate.

First of all, we recall the following result from \cite{dSPinvol4Symp}:

\begin{prop}[Corollary 3.3 in \cite{dSPinvol4Symp}]\label{prop:existlagrangian}
Let $(s,u)$ be an s-pair such that $u$ has no Jordan cell of odd size for an eigenvalue in $\{\pm 1\}$.
Then there exists an $s$-Lagrangian $\mathcal{L}$ such that $u(\mathcal{L}) \cap \mathcal{L}=\{0\}$
and the bilinear form $s_{u,\mathcal{L}} : (x,y) \mapsto s(x,u(y))$ is symmetric and nondegenerate.
\end{prop}

Now, we turn to the case of indecomposable cells of type VI. In \cite{dSPinvol4Symp}, we entirely refrained from tackling them directly,
but here it will be necessary to do so.

\begin{prop}\label{prop:existpresquelagrangian}
Let $(s,u)$ be an s-pair which is an indecomposable cell of type $VI$ with dimension $4n+2$ (where $n \geq 1$) and eigenvalue $\eta$.
Then there exists a $2n$-dimensional totally $s$-singular subspace $W$ such that:
\begin{enumerate}[(i)]
\item $s_{u,W}$ is a hyperbolic bilinear form;
\item $u(W) \cap W=\{0\}$;
\item When denoting by $\pi$ the orthogonal projection onto $(W +u(W))^{\bot_s}$, the endomorphism
$x \mapsto \pi(u(x))$ of $(W +u(W))^{\bot_s}$ equals $\eta\id$.
\end{enumerate}
\end{prop}

One thing must be pointed out. In the situation of Proposition \ref{prop:existpresquelagrangian}, there can be no Lagrangian $\calL$ such that $s_{u,\calL}$
is symmetric and nondegenerate: indeed, it can be proved that the existence of such a Lagrangian would yield that the invariant factors of $u$
are palindromials of even degree, which is not the case here.
So, the result we have just stated is essentially optimal for indecomposable cells of type VI.

\begin{proof}
Multiplying $u$ by $\eta$ if necessary, we can assume that $\eta=1$.

We use an algebraic model for $(s,u)$. Consider the ring $\F[t,t^{-1}]$ of all Laurent polynomials with one indeterminate $t$ and coefficients in $\F$, the quotient ring
$$R:=\F[t,t^{-1}]/\bigl((t-1)^{2n+1}\bigr)$$
and the involution $x \mapsto x^\bullet$ of $R$ that takes the coset $\lambda$ of $t$ to its inverse (it is well defined because the ideal generated by
$(t-1)^{2n+1}$ is invariant under the involution of $\F[t,t^{-1}]$ that takes $t$ to $t^{-1}$).
The eigenspaces of this involution are denoted by
$$H:=\{x \in R : \; x^\bullet=x\} \quad \text{(the set of all Hermitian elements)}$$
and
$$S:=\{x \in R : \; x^\bullet=-x\} \quad \text{(the set of all skew-Hermitian elements).}$$
One sees that $H$ is the subalgebra generated by
$\lambda+\lambda^\bullet$, and it is naturally isomorphic to $\F[t]/(t-2)^{n+1}$ through the isomorphism 
$$\varphi : H \overset{\simeq}{\longrightarrow} \F[t]/(t-2)^{n+1}$$
that takes $\lambda+\lambda^{-1}$ to the coset of $t$ mod $(t-2)^{n+1}$. In particular, $\dim_\F H=n+1$, and hence $\dim_\H S=n$.

Note that $xy \in S$ for all $x \in H$ and $y \in S$.
For $x \in R$, we denote by
$$x_h:=\frac{x+x^\bullet}{2} \quad \text{and} \quad x_s:=\frac{x-x^\bullet}{2}$$
its Hermitian and skew-Hermitian part, respectively.
It will be useful to note that the kernel of
$$\kappa : x \in R \mapsto (\lambda-\lambda^{-1})\,x$$
is the $1$-dimensional $\F$-linear subspace spanned by
the Hermitian element
$$\omega:=(\lambda+\lambda^{-1}-2)^n;$$
Moreover $\kappa(H)=S$: indeed $\kappa(H)\subset S$, whereas $\dim H=\dim S+1$ and $\dim (\Ker \kappa)=1$.

Next, we define $e$ as the $\F$-linear form on $\F[t]/(t-2)^{n+1}$ that takes the coset of
$(t-2)^k$ to $0$ for all $k \in \lcro 0,n-1\rcro$, and that takes the coset of $(t-2)^n$ to $1$.
One checks that $(x,y) \mapsto e(xy)$ is a nondegenerate bilinear form on the $\F$-vector space $\F[t]/(t-2)^{n+1}$.

It follows that $f : (x,y) \mapsto e(\varphi((xy)_h))$ is a symmetric bilinear form on the $\F$-vector space $R$,
and we shall now see that it is nondegenerate.
\begin{itemize}
\item First of all, note that $f(x,y)=0$ for all $(x,y) \in H \times S$.
\item Let $x \in H$ and assume that $f(x,y)=0$ for all $y \in H$. Then, for all $y \in H$ we see that
$e(\varphi(x)\varphi(y))=0$ (because $xy \in H$ and hence $(xy)_h=xy$),
and since $e$ is nondegenerate it follows that $\varphi(x)=0$ and hence $x=0$.
\item Let $x \in S$, and assume that $f(x,y)=0$ for all $y \in S$. Since $\kappa(H)=S$, we have $x=(\lambda-\lambda^{-1})\, x_1$
for some $x_1 \in H$, we deduce that $e(\varphi((\lambda-\lambda^{-1})^2 x_1y_1))=0$ for all $y_1 \in H$, and hence
$(\lambda-\lambda^{-1})^2 x_1=0$. Hence $x \in \Ker \kappa$. But as noted in the above this shows that $x \in H$, whence $x=0$ because $x \in S$.
\end{itemize}
From the above three points, it is easily concluded that $f$ is nondegenerate.
Note also, for the remainder of the proof, that $S=H^{\bot_f}$ as a consequence of the above.

Next, on the $\F$-vector space $R^2$, we consider the symplectic form
$$s' : ((x,y),(x',y')) \longmapsto f(x,y')-f(x',y)$$
and the endomorphism
$$u' : (x,y) \mapsto (\lambda x,\lambda^{-1}y).$$
It is clear that $(s',u')$ is an s-pair and that $u'$ has exactly two Jordan cells attached to the eigenvalue $1$, both with size $2n+1$
(consider the induced endomorphisms on the subspaces $R \times \{0\}$ and $\{0\} \times R$).
Hence $(s,u)$ is isometric to $(s',u')$, and in the remainder of the proof we will simply assume that $(s,u)=(s',u')$.

Note that $H \times S$ is totally $s$-singular.
Now, take
$$H_0:=\{x \in H : \; f(x,1_R)=0\}=\{1_R\}^{\bot_f} \cap H \quad \text{and} \quad
W:=H_0 \times S.$$
We will prove that $W$ has the required properties.
We readily have that $W \subseteq H \times S$ is totally $s$-singular, and it is clear that $\dim W=2n$. Next, we prove that
$s_{u,W}$ is hyperbolic. And first of all we prove that it is symmetric. This amounts to proving that $u+u^{-1}$ maps $W$
into its $s$-orthogonal. Setting $\alpha:=\lambda+\lambda^{-1}$, we see that
$(u+u^{-1})(x,y)=(\alpha x,\alpha y) \in H \times S$ for all $(x,y)\in W$, and hence both $W$ and $(u+u^{-1})(W)$ are included in the totally $s$-singular
subspace $H \times S$, which yields that $s_{u,W}$ is symmetric.

Next, let $(x,y)$ belong to the radical of $s_{u,W}$.
\begin{itemize}
\item For all $z \in S$, we write $s((0,z),u(x,y))=0$. We deduce that $\lambda x \in S^{\bot_f}=H$,
Hence $\lambda^\bullet x^\bullet=\lambda x$, leading to $(\lambda-\lambda^{-1})\, x=0$.
But then $x=\theta \omega$ for some $\theta \in \F$, and $f(1_R,x)=\theta$ by the definition of $e$. Since $x \in H_0$, this leads to $\theta=0$
and hence to $x=0$.
\item For all $z \in H_0$, we write $s((z,0),u(x,y))=0$. It follows that $\lambda^\bullet y \in (H_0)^{\bot_f}$,
and hence $\lambda^\bullet y=x_1+x_2$ where $x_1 \in H \cap (H_0)^{\bot_f}=\F 1_R$ and $x_2 \in S$.
Applying the involution and summing, we get $(\lambda^\bullet-\lambda)\, y=2 x_1$.
Since $\lambda^\bullet-\lambda$ is not invertible in $R$, this equality yields $x_1=0$.
Then $(\lambda-\lambda^{-1})\, y=0$, and we conclude that $y \in \F \omega \cap S=\{0\}$.
\end{itemize}
Hence $(x,y)=(0,0)$, and we conclude that $s_{u,W}$ is nondegenerate.

Noting that $\dim S=n=\frac{\dim W}{2}$ and that $\{0\} \times S$ is obviously totally isotropic for $s_{u,W}$, we gather that $s_{u,W}$ is hyperbolic.

Next, we prove that $u(W) \cap W=\{0\}$, which is equivalent to proving that $H_0 \cap \lambda H_0=\{0\}$ and $S \cap \lambda S=\{0\}$.
\begin{itemize}
\item Let $x \in H_0 \cap \lambda H_0$. Then $\lambda^\bullet x \in H$ so $\lambda^\bullet x=(\lambda^\bullet x)^\bullet=\lambda x$. Once more this leads to $x \in \F \omega$ and then to $x=0$ because $H_0 \cap \F \omega=\{0\}$.

\item Let $x \in S \cap \lambda S$. Then  $\lambda^\bullet x \in S$ so $\lambda^\bullet x=-(\lambda^\bullet x)^\bullet=\lambda x$, and as before we get $x=0$ directly
because $S \cap \F \omega=\{0\}$.
\end{itemize}

It remains to prove statement (c).
To do so, it suffices to prove that the bilinear mapping
$$(x,y)\in (R^2)^2 \longmapsto s(x,(u-\id)(y))$$
vanishes everywhere on $((W+u(W))^{\bot_s})^2$.
Yet obviously
$$(W+u(W))^{\bot_s}=(S^{\bot_f} \cap \lambda S^{\bot_f}) \times \left(H_0^{\bot_f} \cap \lambda^{-1} H_0^{\bot_f}\right)$$
and so it will suffice to see that $e(\varphi(((\lambda-1)xy)_h)=0$ for all $x \in S^{\bot_f} \cap \lambda S^{\bot_f}$
and $y \in R$. But this is easy: let indeed $x \in  S^{\bot_f} \cap \lambda S^{\bot_f}=H \cap \lambda H$.
Then $\lambda^\bullet x \in H$ and $x \in H$, whence $\lambda^\bullet x=\lambda x$. Once more this leads to
$x \in \F \omega$ and hence to $(\lambda-1)x=0$; we conclude that $e(\varphi((\lambda-1)xy)_h)=0$ for all $y \in R$.
This concludes the proof.
\end{proof}

\section{Dimensions that are multiples of $4$}\label{section:dimmultiple4}

Here, we let $(s,u)$ be an s-pair whose dimension $n$ is a multiple of $4$.
Our aim is to prove that $u$ is $4$-reflectional in $\Sp(s)$ unless $n=4$ and $|\F|=3$.

The basic idea is to split $u$ into indecomposable cells and then to regroup all the cells that are not of type VI on one side,
and all those of type VI on the other side:
this yields a splitting
$$u=u_r \botoplus u_e$$
with a corresponding splitting $V=V_r \oplus V_e$ of the underlying vector space $V$, in which:
\begin{itemize}
\item $u_r$ has no Jordan cell of odd size for an eigenvalue in $\{\pm 1\}$;
\item $u_e$ is triangularizable and has all its eigenvalues in $\{\pm 1\}$, and \emph{all} its
Jordan cells of odd size.
\end{itemize}
We will say that $u_r$ is a \textbf{regular} part of $u$, and that $u_e$ is an \textbf{exceptional} part of $u$.

\subsection{Basic results on polynomials over finite fields}

\begin{lemma}\label{lemma:evennotpalindromial}
Assume that $\F$ is finite, and let $p \in \F[t]$ be irreducible, monic, and distinct from $t^2+1$.
Then $p$ cannot be both even and a palindromial.
\end{lemma}

\begin{proof}
Assume on the contrary that $p$ is both even and a palindromial, and consider the splitting field $\L:=\F[t]/(p)$
(this is a splitting field because $\F$ is finite, so every finite extension of it is a Galois one).
The Galois group $\Gal(\L/\F)$ is cyclic, so it has at most one element of order $2$. Yet, we have the Galois automorphism
$\sigma_1 \in \Gal(\L/\F)$ that takes the coset $x$ of $t$ to its opposite (because $p$ is even), and the Galois automorphism
$\sigma_2 \in \Gal(\L/\F)$ that takes $x$ to $x^{-1}$ (because $p$ is a palindromial). Note that $\sigma_1 \neq \id$ and $\sigma_2 \neq \id$
(otherwise $x \in \{0,1,-1\}$, which is not allowed).
And $(\sigma_1)^2(x)=x=(\sigma_2)^2(x)$. Hence, $(\sigma_1)^2=(\sigma_2)^2=\id$, to the effect that $\sigma_1=\sigma_2$.
It would follow that $x^2=-1$, and as $\deg(p) \geq 2$ this would lead to $p=t^2+1$.
\end{proof}

\begin{lemma}\label{lemma:evenirrpolynomial}
Let $n \geq 1$ be a positive integer. Assume that $\F$ is finite.
Then there exists a polynomial $p \in \Irr(\F)$ that is even and has degree $2n$.
If in addition $|\F|>3$ then $p$ can be taken to be different from $t^2+1$.
\end{lemma}

Note that if $p \neq t^2+1$ then Lemma \ref{lemma:evennotpalindromial} shows that $p$ cannot be a palindromial.

\begin{proof}
We can take a finite extension $\L$ of $\F$ of degree $2n$. There is a unique non-identity involution $\sigma$ of $\L$, and we
consider the subfield $\K:=\{x \in \L :\; \sigma(x)=x\}$.
Its group $\K^\times$ of units is cyclic with even order, and we can pick a generator $\omega$ of it.
Then we know that $\L$ contains a square root $z$ of $\omega$, and $z \not\in \K$ (because $|\K^\times|$ is even).
Hence $\F[z]$ contains every power of $\omega$, and hence it includes $\K$; and since $z \not\in \K$ and $[\L:\K]=2$ we deduce that
$\F[z]=\L$. We deduce that $z$ has degree $2n$ over $\F$.
If we denote by $q$ the minimal polynomial of $\omega$ over $\F$, we know that $q$ has degree $2$
and that $q(t^2)$, which is monic of degree $2n$, annihilates $z$. Hence $p:=q(t^2)$ is the minimal polynomial of $z$, to the effect
that it is irreducible. Obviously, it is even.

Finally, if $n=1$ and $|\F|>3$, we see that $\omega \neq -1$, to the effect that $p \neq t^2+1$.
\end{proof}

\subsection{Preliminary work in dimension $4$}

\begin{lemma}\label{lemma:(t^2+1)2}
Assume that $\F$ is finite and $|\F|>3$.
Let $(s,u)$ be an s-pair in which $u$ is $\calC\left((t^2+1)^2\right)$.
Then $u$ is $3$-reflectional in $\Sp(s)$.
\end{lemma}

Here, the assumption that $|\F|>3$ is critical, and the result fails for $\F_3$, which will be seen in Section \ref{section:F3}.

\begin{proof}
We choose an arbitrary $h \in \F \setminus \{0,2,-2\}$. One checks that
the matrix
$$M:=\begin{bmatrix}
0 & -I_2 \\
I_2 & h I_2
\end{bmatrix}$$
belongs to $\Sp_4(\F)$; it has two invariant factors, all equal to
$t^2-h t+1$, a polynomial with no root in $\{-1,1\}$.

As the field is finite, we deduce from Nielsen's theorem that $M$ is $2$-reflectional in $\Sp_4(\F)$.
Let us consider the involution
$$A:=\begin{bmatrix}
0 & K \\
K^{-1} & 0
\end{bmatrix}\in \Sp_4(\F), \quad \text{where}\; K:=\begin{bmatrix}
0 & -1 \\
1 & 0
\end{bmatrix}.$$
One sees that
$$AM=\begin{bmatrix}
K & h K \\
0 & K
\end{bmatrix} \quad \text{and then} \quad
(AM)^2=\begin{bmatrix}
-I_2 & -2hI_2 \\
0 & -I_2
\end{bmatrix}.$$
Hence $(t^2+1)^2$ annihilates $AM$ but $t^2+1$ does not.

The minimal polynomial of $AM$ is a palindromial with no root in $\{-1,1\}$, so its degree is even, and we conclude that
it equals $(t^2+1)^2$.

Finally, since $\F$ is finite and $AM$ has no eigenvalue in $\{-1,1\}$, the conjugacy class of $AM$ in $\Sp_4(\F)$
is determined by the Jordan numbers of $AM$. Hence, every $\calC\left((t^2+1)^2\right)$-symplectic automorphism
is represented by $AM$ in a well-chosen symplectic basis. As $AM$ is $3$-reflectional in $\Sp_4(\F)$, we conclude that $u$ is $3$-reflectional in $\Sp(s)$.
\end{proof}

\subsection{The regular case}

Here, we will consider the easiest case: the one where $u$ has no Jordan cell of odd size for an eigenvalue in $\{-1,1\}$.
This case is dealt with thanks to the space-pullback method, but there are additional complexities due to the failure of
Lemma \ref{lemma:(t^2+1)2} for $\F_3$.

\begin{prop}\label{prop:nrmultiplede4}
Assume that $\F$ is finite.
Let $(s,u)$ be an s-pair whose dimension $n$ is a multiple of $4$. Assume that $u$ has no Jordan cell of odd size for an eigenvalue in $\{\pm 1\}$.
Then $u$ is $4$-reflectional in $\Sp(s)$ unless $n=4$ and $|\F|=3$.
\end{prop}

\begin{proof}
We discard the special case where $n=4$ and $|\F|=3$.

Denote by $V$ the underlying vector space of $(s,u)$.
To start with, Proposition \ref{prop:existlagrangian} yields a Lagrangian $\calL$ of $(V,s)$ such that $s_{u,\calL}$ is nondegenerate and symmetric, and
$\calL \cap u(\calL)=\{0\}$. We will use a space-pullback thanks to $\calL$, and it remains to make a careful choice of
a symplectic form on $\calL$, which amounts to choose an $s_{u,\calL}$-skew-selfadjoint automorphism $v$ of $\calL$
(then the symplectic form we take is $(x,y) \mapsto s_{u,\calL}(x,v^{-1}(y))$).
By such a space-pullback, $u$ will be seen to be i-adjacent to some $u' \in \Sp(s)$ that leaves $\calL$ invariant and for which
$u'_\calL=v$.

The favorable case is the one where $s_{u,\calL}$ is non-hyperbolic. In that case, we can use Lemma \ref{lemma:evenirrpolynomial}
to find an even polynomial $p \in \Irr(\F)$ with degree $\frac{n}{2}$ and $p \neq t^2+1$ (note how this uses the assumption that
$n>4$ or $|\F|>3$). Then $p$ is not a palindromial, so it is relatively prime with $p^\sharp$.
By Lemma \ref{lemma:adaptlemmairr} (using the fact that two nondegenerate and nonhyperbolic symmetric bilinear forms over a finite field -- of characteristic other than $2$ -- are equivalent whenever they have the same rank) we find a $\calC(p)$-automorphism $v$ of $\calL$ such that $v$ is $s_{u,\calL}$-skew-selfadjoint.
Hence, in using $\calL$ and $v$ for space-pullbacking, we find that $u$ is i-adjacent to some $u' \in \Sp(s)$ that leaves $\calL$ invariant and for which
$u'_\calL$ is $\calC(p)$. And then, since $p$ is relatively prime with $p^\sharp$ we deduce from Lemma \ref{lemma:recog} that
$u'$ is a symplectic extension of $u'_\calL$. And by Proposition \ref{prop:even3refl} the symplectic transformation $u'$ is $3$-reflectional in $\Sp(s)$.
Hence $u$ is $4$-reflectional in $\Sp(s)$.

In the remainder of the proof, we assume that $s_{u,\calL}$ is hyperbolic. In this unfavorable case, we need to split the discussion into as many as four subcases.

\vskip 3mm
\noindent \textbf{Case 1: $|\F|>3$ and $t^2+1$ splits over $\F$.} \\
We note that $\frac{n}{2}$ is even. By Lemma \ref{lemma:adaptlemmasplit}, we can apply the space-pullback technique
to obtain that $u$ is i-adjacent to some $u' \in \Sp(s)$ that leaves $\calL$ invariant and such that
$(u')_\calL$ is diagonalisable with characteristic polynomial $(t^2+1)^{n/4}$ (take $\lambda$ as a root of $t^2+1$).
It easily follows that $(t^2+1)^2$ annihilates $u'$. Hence, $u'$ is an orthogonal direct sum of indecomposable
cells with minimal polynomial $(t^2+1)^2$ and of an even number of indecomposable cells with minimal polynomial $t^2+1$.
Yet, by Nielsen's theorem, any orthogonal direct sum of an even number of indecomposable cells with minimal polynomial $t^2+1$
is $2$-reflectional. Hence, by Lemma \ref{lemma:(t^2+1)2}, we see that $u'$ is $3$-reflectional in $\Sp(s)$, and hence $u$ is $4$-reflectional in
$\Sp(s)$.

\vskip 3mm
\noindent \textbf{Case 2: $|\F|>3$ and $t^2+1$ is irreducible over $\F$.} \\
In particular $|\F|>5$, so we can choose $\lambda \in \F \setminus \{0\}$ such that $\lambda^4 \neq 1$.
Then by space-pullbacking, we find that $u$ is i-adjacent to a symplectic transformation $u'$ that leaves $\calL$ invariant and whose
induced automorphism of $\calL$ is a direct sum of $\calC\left((t-\lambda)(t+\lambda)\right)$-automorphisms.
As $(t-\lambda)(t+\lambda)$ is relatively prime with its reciprocal polynomial, we find that $u'$ is a symplectic extension of $(u')_{\calL}$.
And then we deduce that $u'$ is the orthogonal direct sum of symplectic extensions of $\calC(t^2-\lambda^2)$-automorphisms.
Hence, by Proposition \ref{prop:even3refl} we see that $u'$ is $3$-reflectional in $\Sp(s)$, and we conclude that $u$ is $4$-reflectional in $\Sp(s)$.

\vskip 3mm
\noindent \textbf{Case 3: $|\F|=3$ and $n$ is a multiple of $8$.} \\
We know from Lemma \ref{lemma:adaptlemmairr}
that there exists a pair $(b,w)$ that consists of a non-hyperbolic nondegenerate symmetric bilinear form $b$
of rank $2$ and of a $b$-skew-symmetric automorphism $w$ with minimal polynomial $t^2+1$.
By taking an orthogonal direct sum of $\frac{n}{4}$ copies of $(b,w)$, we see that the resulting pair $(\widetilde{b},\widetilde{v})$
is such that $\widetilde{b}$ is hyperbolic (because $\frac{n}{4}$ is even) and that $\widetilde{w}$ has minimal polynomial $t^2+1$.
Hence, by using an isometry, we find that there exists an $s_{u,\calL}$-skew-symmetric automorphism $w$ with minimal polynomial $t^2+1$.
Applying the space-pullback technique with $\calL$ and $v$, we deduce that $u$ is i-adjacent to some $u' \in \Sp(s)$
that leaves $\calL$ invariant and such that $u'_\calL$ has minimal polynomial $t^2+1$.
Hence $u'$ is clearly annihilated by $(t^2+1)^2$, and we deduce that it splits into the orthogonal direct sum of
$n_1$ indecomposable cells with minimal polynomial $(t^2+1)^2$, and of $n_2$
indecomposable cells with minimal polynomial $t^2+1$, for some integers $n_1 \geq 0$ and $n_2 \geq 0$ that satisfy $4n_1+2n_2=n$.
Hence $2n_1+n_2$ is a multiple of $4$, and we deduce that $n_2$ is even and that $n_1$ and $\frac{n_2}{2}$ have the same parity.
If $n_1$ is even then $u'$ directly satisfies the conditions of Nielsen's theorem. If $n_1$ is odd, we resplit $u'$ into
$(v_1 \botoplus v_2) \botoplus v_3$ where:
\begin{itemize}
\item $v_1$ is an indecomposable cell with minimal polynomial $(t^2+1)^2$;
\item $v_2$ is the orthogonal direct sum of two indecomposable cells with minimal polynomial $t^2+1$;
\item $v_3$ is the orthogonal direct sum of an even number of
indecomposable cells with minimal polynomial $(t^2+1)^2$ and of an even number of indecomposable cells with minimal polynomial $t^2+1$.
\end{itemize}
The difficulty here is that $v_1$ is not $3$-reflectional, and the trick
consists in using Lemma \ref{lemma:from(t2+1)2tot^4+1}: it shows indeed that $v_1$ is i-adjacent to some
$\calC(t^4+1)$-transformation $v'_1$ (with the same underlying space).
Now, the same can be said of $v_2$, thanks to Lemma \ref{lemma:cyclicfit}: indeed, as $\F$ is finite we see that
$v_2$ is actually a symplectic extension of a $\calC(t^2+1)$-automorphism, and hence $v_2$
is i-adjacent to a $\calC(t^2-t-1)$-automorphism $v'_2$, and just like in the end of the proof of Lemma \ref{lemma:from(t2+1)2tot^4+1}
we find that $v'_2$ is $\calC(t^4+1)$. Hence $u'$ is i-adjacent to $(v'_1 \botoplus v'_2)\botoplus v_3$,
and by Nielsen's theorem for finite fields the latter is $2$-reflectional. Hence $u$ is $4$-reflectional in $\Sp(s)$.

\vskip 3mm
\noindent
\textbf{Case 4: $|\F|=3$ and $n=8m+4$ for some $m \geq 1$.} \\
By Lemma \ref{lemma:evenirrpolynomial}, we can choose an \emph{even} $q \in \Irr(\F)$ with degree $4m$. Note that $q(0)=\pm 1$ because $|\F|=3$.
By Lemma \ref{lemma:evennotpalindromial}, $q$ is not a palindromial, and hence it is relatively prime with $q^\sharp$. And obviously $q$ and $q^\sharp$ are relatively prime with $t^2+1$.
By Lemma \ref{lemma:adaptlemmairr}, we can find pairs $(b_1,v_1)$ and $(b_2,v_2)$ in which
$b_1$ and $b_2$ are non-hyperbolic nondegenerate symmetric bilinear forms, $v_1$ is $b_1$-skew-selfadjoint and $\calC(t^2+1)$,
and $v_2$ is $b_2$-skew-selfadjoint and $\calC(q)$. Taking the orthogonal direct sum yields
a pair $(b,w)$ in which $b$ is a hyperbolic bilinear form and $w$ is a $b$-skew-selfadjoint $\calC((t^2+1)\,q)$-automorphism.
Hence, there exists an $s_{u,\calL}$-skew-selfadjoint $\calC((t^2+1)\,q)$-automorphism $v$ of $\calL$.
Applying the space-pullback technique, we deduce that $u$ is i-adjacent to some $u' \in \Sp(s)$ that leaves $\calL$ invariant and such that
$u'_\calL=v$. Then the characteristic polynomial of $u'$ equals $qq^\sharp (t^2+1)^2$. By decomposing $u'$ into indecomposable cells, we find a splitting
$u'=u'_1 \botoplus u'_2$ in which $u'_1$ is a symplectic extension of a $\calC(q)$-automorphism, and
$u'_2$ is either $\calC\left((t^2+1)^2\right)$ or the orthogonal direct sum of two $\calC(t^2+1)$-symplectic transformations.

Assume first that $u'_2$ has minimal polynomial $t^2+1$. Then by Nielsen's theorem it is $2$-reflectional, whereas Proposition \ref{prop:even3refl}
yields that $u'_1$ is $3$-reflectional. Hence $u'$ is $3$-reflectional, and we deduce that $u$ is $4$-reflectional.

Assume finally that $u'_2$ is $\calC\left((t^2+1)^2\right)$. Then by Lemma \ref{lemma:from(t2+1)2tot^4+1} it is i-adjacent to
some $\calC(t^4+1)$-symplectic transformation $u''_2$. Besides, by Lemma \ref{lemma:cyclicfit}
$u'_1$ is i-adjacent to a symplectic extension $u''_1$ of a $\calC\left((t-1)^{4m-3}(t-q(0))(t^2-t-1)\right)$-automorphism $v'$.
Using Lemma \ref{lemma:splitsymplecticextension}, we can further resplit $v'=j_1\botoplus j_2$
where $j_1$ is a symplectic extension of a $\calC\left((t-1)^{4m-3}(t-q(0))\right)$-automorphism and $j_2$ is
a symplectic extension of a $\calC(t^2-t-1)$-automorphism.
Hence $j_1$ is $2$-reflectional by Nielsen's theorem, whereas $j_2$ is $\calC(t^4+1)$.
Hence $u'$ is i-adjacent to $j_1 \botoplus (j_2\botoplus u''_2)$, and $j_2 \botoplus u''_2$ is $2$-reflectional by Nielsen's theorem.
Hence $j_1 \botoplus (j_2\botoplus u''_2)$ is $2$-reflectional in $\Sp(s)$, and we conclude that $u$ is $4$-reflectional in $\Sp(s)$.
\end{proof}

In the next two parts, we examine the situation where the dimension of a regular part is not a multiple of $4$.
This requires that we pair this regular part with a cell of type VI, and hence we will start by examining the
most simple situation, in which the regular part has dimension $2$.

\subsection{Pairing a cyclic cell of dimension $2$ with an indecomposable cell of type VI}\label{section:multiple4I}

We start with the case of finite fields with more than $3$ elements.

\begin{lemma}\label{lemma:cyclic2+typeVI}
Assume that $\F$ is finite and $|\F|>3$.
Let $(s,u)$ be an s-pair in which $u=u_1 \botoplus u_2$, where
$u_1$ is cyclic with minimal polynomial $p$ of degree $2$, and $u_2$ has exactly two Jordan cells, both of odd size $2n+1$,
for an eigenvalue in $\{\pm 1\}$.
Then:
\begin{enumerate}[(a)]
\item $u$ is $4$-reflectional in $\Sp(s)$;
\item If $n=0$ then $u$ is i-adjacent to a $3$-reflectional $\calC(p(t^2))$-symplectic transformation $u'$.
\end{enumerate}
\end{lemma}

We start by recalling the following result from \cite{dSPinvol4Symp}:

\begin{lemma}[Lemma 5.2 in \cite{dSPinvol4Symp}]\label{lemma:splitquadratic}
Let $\lambda \in \F \setminus \{0,1,-1\}$.
Let $v$ be a cyclic automorphism of a vector space $P$ of dimension $2$, with determinant $1$.
Then $v$ is the product of two $\calC\left((t-\lambda)(t-\lambda^{-1})\right)$ endomorphisms $v_1$ and $v_2$ of $P$.
\end{lemma}

\begin{proof}[Proof of Lemma \ref{lemma:cyclic2+typeVI}]
The assumption $|\F|>3$ allows us to find $\lambda \in \F \setminus \{0\}$ such that $\lambda \neq \lambda^{-1}$, and throughout the proof we fix such a scalar
$\lambda$. Replacing $u$ with $-u$ if necessary, we can assume that $1$ is the sole eigenvalue of $u_2$.

Denote by $V_1$ and $V_2$ the respective underlying vector spaces of $u_1$ and $u_2$.
For some transverse lagrangians $\calL$ and $\calL'$ of $V_2$, the transformation $u_2$ is the symplectic extension
$s_{\calL'}(v)$ for some automorphism $v$ of $\calL$ which is a Jordan cell of size $2n+1$ for the eigenvalue $1$.
Let us take a basis $(e_1,\dots,e_{2n+1})$ of $\calL$ in which $v(e_i)=e_i+e_{i-1}$ for all $i \in \lcro 1,2n+1\rcro$
(where we convene that $e_0=0$).

By Lemma \ref{lemma:splitquadratic}, the automorphism $u_1$ is the product of two $\calC\left((t-\lambda)(t-\lambda^{-1})\right)$-automorphisms $u'_1$ and $u''_1$ of $V_1$,
and hence both $u'_1$ and $u''_1$ are symplectic. Denote by $v'$ the automorphism of $\calL$ such that $v'(e_i)=e_i$ for all $i \in \lcro 1,2n+1\rcro \setminus \{n\}$,
and $v'(e_{n+1})=\lambda^{-1} e_{n+1}$; set then $u':=u'_1 \botoplus s_{\calL'}(v')$ and $u'':=(u')^{-1} u=u''_1 \botoplus s_{\calL'}((v')^{-1} v)$.
We shall see that both $u'$ and $u''$ are $2$-reflectional in $\Sp(s)$, which will prove point (a).

By Lemma \ref{lemma:splitsymplecticextension}, we see that $u'=u'_1 \botoplus w' \botoplus \id$ where $w'$ is a symplectic extension of $\lambda^{-1} \id$
(defined on a $1$-dimensional Lagrangian), and hence $w'$ is $\calC\left((t-\lambda)(t-\lambda^{-1})\right)$.
Hence $u'_1 \botoplus w'$ is diagonalizable with spectrum $\{\lambda,-\lambda\}$ and all its Jordan numbers are even. By Nielsen's theorem,
$u'_1 \botoplus w'$ is $2$-reflectional in the corresponding symplectic group, and hence $u'$ is $2$-reflectional in $\Sp(s)$.

Finally, we note that $(v')^{-1} v$ is obviously $\calC\left((t-1)^{2n}(t-\lambda)\right)$ (with cyclic vector $e_{2n+1}$),
and hence it is the direct sum of a $\calC(t-\lambda)$-automorphism $w_2$ and of a $\calC\left((t-1)^{2n}\right)$ cyclic automorphism $j_2$.
Lemma \ref{lemma:splitsymplecticextension} yields a splitting $u''=(u''_1 \botoplus w'_2) \botoplus j'_2$
where $w'_2$ and $j'_2$ are respective symplectic extensions of $w_2$ and $j_2$.
With the same line of reasoning as for $u'$, we use Nielsen's theorem to see that
$u''_1 \botoplus w'_2$ is $2$-reflectional (in the corresponding symplectic group), and so is $j'_2$
because Wonenburger's theorem \cite{Wonenburger} shows that $j_2$ is $2$-reflectional in the corresponding general linear group. Hence $u''$ is $2$-reflectional in $\Sp(s)$, which completes the proof of point (a).

It remains to prove point (b). So let us now assume that $n=0$.
In that case, we choose symplectic bases $(e_1,f_1)$ and $(e_2,f_2)$ in which the respective matrices of $u_1$ and $u_2$ are $C$ and $I_2$,
where $C$ is cyclic with minimal polynomial $p$.
Let us consider the endomorphism $\alpha$ of the underlying space $V$ of $(s,u)$ whose matrix in the basis $(e_1,f_1,e_2,f_2)$ equals
$\begin{bmatrix}
0 & I_2 \\
I_2 & 0
\end{bmatrix}$. Clearly, $\alpha \in \Sp(s)$ and $\alpha$ is an involution. The matrix of $\alpha u$ in the permuted basis
$(e_2,f_2,e_1,f_1)$ equals
$\begin{bmatrix}
0 & C \\
I_2 & 0
\end{bmatrix}$, and by Lemma 14 of \cite{dSP2LC} (in which the condition that $\alpha \neq 0$ and $\beta \neq 0$ is useless) it is cyclic with minimal polynomial $p(t^2)$.
Now, as in the above we can split $C=A_1 A_2$ where both $A_1$ and $A_2$ are cyclic with minimal polynomial $(t-\lambda)(t-\lambda^{-1})$.
Denote by $\beta$ the automorphism of $V$ whose matrix in the basis $(e_1,f_1,e_2,f_2)$ equals
$\begin{bmatrix}
0 & A_1^{-1} \\
A_1 & 0
\end{bmatrix}$. Again, one sees that $\beta$ belongs to $\Sp(s)$ (because $\det A_1=1$) and that it is an involution.
The matrix of $\beta (\alpha u)$ in $(e_1,f_1,e_2,f_2)$ then equals
$\begin{bmatrix}
A_2 & 0 \\
0 & A_1
\end{bmatrix}$, and as in the above we use Nielsen's theorem to see that it is $2$-reflectional. Hence $\alpha u$ is $3$-reflectional in $\Sp(s)$,
which proves point (b).
\end{proof}

With the same line of reasoning as in the start of the proof of point (b), we also obtain:

\begin{lemma}\label{lemma:cyclic2+typeVIbis}
Assume that $\F$ is finite.
Let $(s,u)$ be an s-pair in which $u=u_1 \botoplus u_2$, where
$u_1$ is cyclic with minimal polynomial $p$ of degree $2$, and $u_2$ is the identity on a $2$-dimensional space.
Then $u$ is i-adjacent to a symplectic transformation $u'$ with minimal polynomial $p(t^2)$.
\end{lemma}

We move forward with a result that holds regardlessly of the field $\F$ under consideration (but with $\chi(\F) \neq 2$, of course).

\begin{lemma}\label{lemma:(t+1)2+identity}
Let $(s,u)$ be an s-pair with a splitting $u=u_r \botoplus u_e$ where $u_r$ is $\calC((t+1)^2)$, and $u_e$
is the identity of a $2$-dimensional subspace. Then $u$ is $3$-reflectional in $\Sp(s)$.
\end{lemma}

\begin{proof}
Let $\varepsilon \in \F \setminus \{0\}$.

We will work backwards. Denote by $s'$ the standard symplectic form on $\F^4$, and put $D:=\begin{bmatrix}
-1 & 1 \\
0 & 1
\end{bmatrix}$ and
$\Delta:=\begin{bmatrix}
0 & 1 \\
1 & 0
\end{bmatrix}$.
Note that $\Delta$ is hyperbolic as a symmetric matrix.
The matrices
$$J:=\begin{bmatrix}
D & 0 \\
0 & D^\sharp
\end{bmatrix} \quad \text{and} \quad
V:=\begin{bmatrix}
I_2 & \varepsilon \Delta \\
0 & I_2
\end{bmatrix}$$
belong to $\Sp_4(\F)$. The first one is involutory because $D^2=I_2$, and the second one is $2$-reflectional by Nielsen's theorem, because
its minimal polynomial is $(t-1)^2$ and its Wall invariant attached to $(t-1,2)$ has Gram matrix $-\varepsilon \Delta$, which is hyperbolic.
Hence the product $JV$ is $3$-reflectional. We find
$$JV=\begin{bmatrix}
D & \varepsilon D\Delta \\
0 & D^\sharp
\end{bmatrix}.$$
Obviously, $(t+1)^2(t-1)^2$ is the characteristic polynomial of $JV$. Moreover,
$$JV-I_4=\begin{bmatrix}
-2 & 1 & \varepsilon & -\varepsilon \\
0 & 0 & \varepsilon & 0 \\
0 & 0 & -2 & 0 \\
0 & 0 & 1 & 0
\end{bmatrix}.$$
Clearly $\rk(JV-I_4)=2$. Moreover
$$(JV)^2-I_4=\begin{bmatrix}
0_2 & \varepsilon(\Delta +D\Delta D^\sharp) \\
0_2 & 0_2
\end{bmatrix},$$
where $\Delta+D\Delta D^\sharp=\begin{bmatrix}
-2 & 0 \\
0 & 0
\end{bmatrix}$ has rank $1$.
Hence the symplectic automorphism $u'$ of $\F^4$ that is represented by $JV$ in the standard basis
is the direct sum of the identity on a $2$-dimensional space and of a $\calC((t+1)^2)$-automorphism.

In order to conclude, it will suffice to adjust $\varepsilon$ so that $(s',u')$ and $(s,u)$ have equivalent Wall invariants attached to $(t+1,2)$.
Yet the Wall invariant $(s,u')_{t+1,2}$ is equivalent to the regular part of the symmetric bilinear form $s_{u'}$,
and this regular part has rank $1$. Ditto for $u$.
To identify $(s,u')_{t+1,2}$ up to equivalence, it suffices to compute a non-zero value of the quadratic form attached to $s_{u'}$.
Denoting by $e_1,\dots,e_4$ the standard basis of $\F^4$,
we see that $s_{u'}(e_3,e_3)=s(e_3,\varepsilon e_1+\varepsilon e_2-e_3+e_4)=\varepsilon s(e_3,e_1)=-\varepsilon$, so
we can always choose $\varepsilon$ so that the value $-\varepsilon$ is reached by the quadratic form which is attached to the Wall invariant $(s,u)_{t+1,2}$.
With that choice, $(s',u') \simeq (s,u)$, which proves that $u$ is $3$-reflectional in $\Sp(s)$.
\end{proof}

\begin{lemma}\label{lemma:pairF3eigenvalue}
Assume that $|\F|=3$.
Let $(s,u)$ be an s-pair with a splitting $u=u_r \botoplus u_e$ where $u_r$ is cyclic of rank $2$ with an eigenvalue in $\{\pm 1\}$
and $u_e$ is an indecomposable cell of type VI, with rank $4n+2\geq 2$.
Then $u$ is $4$-reflectional in $\Sp(s)$.
\end{lemma}

\begin{proof}
We use Lemma \ref{lemma:cyclicfit} to work on $u_e$. First of all, we lose no generality in assuming that $-1$ is the eigenvalue of $u_r$
because the problem is unchanged in replacing $u$ with $-u$.

Let us set $p:=t^{2n}+1$ if $n>0$, and $p:=1$ otherwise. Since $2n+1$ is odd, Lemma \ref{lemma:cyclicfit} shows that $u_e$ is i-adjacent to
a symplectic extension $w$ of a $\calC((t-1)\,p)$-automorphism.
Hence $u$ is i-adjacent to $u_r \botoplus w$.
Thanks to Lemma \ref{lemma:splitsymplecticextension}, we can resplit $w =w_1 \botoplus w_2$ where $w_1=\id_P$ for some $2$-dimensional
subspace $P$, and $w_2$ is a symplectic extension of a $\calC(p)$-automorphism.
Since $p$ is a palindromial, Nielsen's theorem shows that $w_2$ is $2$-reflectional, whereas $u_r \botoplus w_1$ is $3$-reflectional by Lemma \ref{lemma:(t+1)2+identity}. Hence $u_r  \botoplus w$ is $3$-reflectional in $\Sp(s)$, and we conclude that $u$ is $4$-reflectional in $\Sp(s)$.
\end{proof}

One case remains to be tackled.

\begin{lemma}\label{lemma:pairF3t^2+1}
Assume that $|\F|=3$.
Let $(s,u)$ be an s-pair with a splitting $u=u_r \botoplus u_e$ where $u_r$ is $\calC(t^2+1)$ and $u_e$
is an indecomposable cell of type VI with rank $4n-2\geq 2$.
Then $u$ is $4$-reflectional in $\Sp(s)$.
\end{lemma}

\begin{proof}
We start with the case $n=1$.
Denote by $V_r$ and $V_e$ the respective underlying spaces of $u_e$ and $u_r$.

Replacing $u$ with $-u$ if necessary, we can assume that the eigenvalue of $u_e$ is $1$.
Because the field $\F$ is finite, the symplectic transformations of a $2$-dimensional symplectic space
that have minimal polynomial $t^2+1$ form a conjugacy class. Hence there is
a symplectic basis $(e_1,e_2)$ of $V_r$ such that $u(e_1)=e_2$ and $u(e_2)=-e_1$.

Let us also take a symplectic basis $(f_1,f_2)$ of $V_e$.
We apply the space-pullback technique with the subspace $W:=\Vect(e_1+f_1,e_2-f_2)$, which is easily seen to be a Lagrangian of $V$.
First of all, we compute that
$$\begin{cases}
s_{u,W}(e_1+f_1,e_1+f_1)& =s(e_1+f_1,e_2+f_1)=s(e_1,e_2)=1 \\
s_{u,W}(e_2-f_2,e_2-f_2)& =s(e_2-f_2,-e_1-f_2)=s(e_2,-e_1)=1 \\
s_{u,W}(e_1+f_1,e_2-f_2)& =s(e_1+f_1,-e_1-f_2)=s(f_1,-f_2)=-1 \\
s_{u,W}(e_2-f_2,e_1+f_1)& =s(e_2-f_2,e_2+f_1)=s(-f_2,f_1)=1.
\end{cases}$$
Note that $s_{u,W}$ is non-symmetric, a rare instance where we use the space-pullback technique with such a bilinear form!

Now, let us take the symplectic form $b$ on $W$ for which $(e_1+f_1,e_2-f_2)$ is a symplectic basis.
Then the endomorphism $v$ of $W$ such that $\forall (x,y)\in W^2, \; s(x,u(y))=b(x,v(y))$
has its matrix in $(e_1+f_1,e_2-f_2)$ equal to
$A=\begin{bmatrix}
-1 & -1 \\
1 & -1
\end{bmatrix}$,
which is cyclic with characteristic polynomial $t^2+2t+2=t^2-t-1$.
Applying the space-pullback technique, we see that $u$ is i-adjacent to some $u' \in \Sp(s)$
that stabilizes $W$ with $u'_W=v$. Since $t^2-t-1$ is relatively prime with its reciprocal polynomial $t^2+t-1$,
we deduce from Lemma \ref{lemma:recog} that $u'$ is a symplectic extension of $v$.
And then Proposition \ref{prop:sympextension3refl} shows that $u'$ is $3$-reflectional in $\Sp(s)$. Therefore $u$ is $4$-reflectional in $\Sp(s)$.

Next, we assume that $n \geq 2$. Here, we will use the Cyclic Adaptation Lemma to
break $u_e$ up. This lemma yields that $u_e$ is i-adjacent to a symplectic extension of
a $\calC((t-1)(t^2-t-1)(t+1)^{2n-4})$-automorphism, which can be resplit into
an orthogonal direct sum $w_1 \botoplus w_2 \botoplus w_3$ where $w_1$ is the identity on a $2$-dimensional space, $w_2$
has characteristic polynomial $t^4+1$, and $w_3$ is a symplectic extension of a $\calC\left((t+1)^{2n-4}\right)$-automorphism.
Hence $u$ is i-adjacent to $(u_r \botoplus w_1) \botoplus w_2 \botoplus w_3$.
In the first part of the proof, we have shown that $u_r \botoplus w_1$ is i-adjacent to a $\calC(t^4+1)$-automorphism $w'_1$,
and hence $u$ is i-adjacent to $(w'_1 \botoplus w_2) \botoplus w_3$. And by Nielsen's theorem both
$w'_1 \botoplus w_2$ and $w_3$ are $2$-reflectional. Hence $u$ is $4$-reflectional.
\end{proof}

We are now able to conclude the present section:

\begin{prop}\label{prop:rank2+typeVI}
Assume that $\F$ is finite. Let $(s,u)$ be an s-pair that splits into the orthogonal direct sum of
a $2$-dimensional s-pair and of an indecomposable s-pair of type VI.
Then $u$ is $4$-reflectional in $\Sp(s)$.
\end{prop}

\begin{proof}
Take the first summand: if it is cyclic, then the result is obtained by Lemma \ref{lemma:cyclic2+typeVI} if
$|\F|>3$, and by combining Lemmas \ref{lemma:pairF3eigenvalue} and \ref{lemma:pairF3t^2+1} otherwise;
if it is not cyclic, then it is indecomposable of type VI, and then we readily deduce from Nielsen's theorem that $u$ is $2$-reflectional in $\Sp(s)$.
\end{proof}

\subsection{A favorable case where $\rk u_r+\rk u_e \geq 12$}

We will obtain the following lemma thanks to the space-pullback technique:

\begin{lemma}\label{lemma:nr+nesup12}
Let $(s,u)$ be an s-pair with a regular/exceptional decomposition $u=u_r \botoplus u_e$
in which $\rk u_r \geq 6$, $\rk u \geq 12$, $\rk u_r$ is not a multiple of $4$, and
$u_e$ is a cell of type VI.
Then $u$ is $4$-reflectional in $\Sp(s)$.
\end{lemma}

\begin{proof}
Denote by $V_r$ and $V_e$ the respective underlying spaces of $u_r$ and $u_e$ and by
$4n_r+2$ and $4n_e+2$ their respective dimensions.

First of all, we use Proposition \ref{prop:existlagrangian} to obtain a Lagrangian $\calL$ of $V_r$ such that
$s_{u,\calL}$ is symmetric and nondegenerate, and $\calL \cap u(\calL)=\{0\}$.
Since $\F$ is finite, we can find a linear hyperplane $H_r$ of $\calL$ such that $s_{u,H_r}$ is regular and non-hyperbolic.

Note that $H_r$ has dimension $2n_r$ and that $u(H_r) \cap H_r=\{0\}$.
Next, by Lemma \ref{prop:existpresquelagrangian}, we can find a totally $s$-singular subspace $H_e$ of $V_e$ such that $s_{u,H_e}$ is hyperbolic,
$H_e \cap u(H_e)=\{0\}$, $\dim H_e=2n_e$, and the bilinear form $(x,y) \mapsto s(x,(u-\id)(y))$ is totally singular on $(H_e +u(H_e))^{\bot_s} \cap V_e$.

We will apply the space-pullback technique with $W:=H_r \botoplus H_e$, which is totally $s$-singular of dimension $2(n_r+n_e)$,
and such that $s_{u,W}$ is regular and non-hyperbolic.
We also note that the bilinear form
$(x,y) \mapsto s(x,u(y))$ on $(W +u(W))^{\bot_s}$ is the orthogonal direct sum of:
\begin{itemize}
\item The bilinear form $(x,y) \mapsto s(x,u(y))$ on $(H_r +u(H_r))^{\bot_s}\cap V_r$, which is defined on a $2$-dimensional space;
\item The bilinear form $(x,y) \mapsto s(x,u(y))$ on $(H_e +u(H_e))^{\bot_s}\cap V_e$, which is defined on a $2$-dimensional space
and coincides on this space with $(x,y) \mapsto s(x,y)$.
\end{itemize}
In applying the space-pullback technique to $u$ and $W$, it follows that the residual
symplectic automorphism $[u]$ of $(W \oplus u(W))^{\bot_s}$
that is attached to $u$ and $W$ is the orthogonal direct sum of two symplectic
transformations of rank $2$, one of which is the identity: hence Proposition \ref{prop:rank2+typeVI} yields
symplectic involutions $i_1,i_2,i_3,i_4$ of $(W \oplus u(W))^{\bot_s}$ such that $[u]=i_1i_2i_3i_4$.

Next, by Lemma \ref{lemma:evenirrpolynomial} we can choose an \emph{even} polynomial $p \in \Irr(\F)$ with degree $2(n_r+n_e)$.
And then Lemma \ref{lemma:adaptlemmairr} yields a symplectic form $b$ on $W$ such that the endomorphism
$v$ of $W$ that satisfies $\forall (x,y)\in W^2, \; b(x,v(y))=s(x,u(y))$ is $\calC(p)$.
In particular, $p$ is not a palindromial (by Lemma \ref{lemma:evennotpalindromial}) but its degree is at least $4$.
As the characteristic polynomial of  $i_1 [u]$ is a palindromial of degree $4$, it is relatively prime with both $p$ and $p^\sharp$.
Hence by combining the space-pullback technique (using $i_1$ as the residual involution),
we obtain that $u$ is i-adjacent to some $u'\in \Sp(s)$ that splits into an orthogonal direct sum  $u'_1 \botoplus u'_2$ where $u'_1$ is a symplectic extension of
$v$, and $u'_2$ is symplectically similar to $i_1[u]$. By the latter, $u'_2$ is $3$-reflectional, whereas Proposition \ref{prop:even3refl}
shows that $u'_1$ is $3$-reflectional. Hence $u'$ is $3$-reflectional in $\Sp(s)$, and we conclude that $u$ is $4$-reflectional in $\Sp(s)$.
\end{proof}

\subsection{The case where $\rk u_r=6$ and $\rk u_e=2$}

We arrive at one of the most difficult situations so far.

\begin{lemma}\label{lemma:rkur=6}
Let $(s,u)$ be an s-pair with a regular/exceptional decomposition $u=u_r \botoplus u_e$
in which $\rk u_r=6$ and $\rk u_e=2$. Then $u$ is $4$-reflectional in $\Sp(s)$.
\end{lemma}

\begin{proof}
Without loss of generality, we can assume that $u_e=\id$.
Denote by $V_r$ and $V_e$ the respective underlying vector spaces of $u_e$ and $u_r$.
By Proposition \ref{prop:existlagrangian}, we can find a Lagrangian $\calL$ of $V_r$ such that
$s_{u,\calL}$ is symmetric and nondegenerate, and $\calL \cap u(\calL)=\{0\}$. Then, since $\F$ is finite we can choose
a linear hyperplane $H$ of $\calL$ such that $s_{u,H}$ is hyperbolic if $t^2+1$ splits over $\F$,
and non-hyperbolic otherwise.
It is critical here to acknowledge that the resulting symplectic automorphism of $(H \oplus u(H))^{\bot_s} \cap V_r$
is cyclic. Indeed, let us take a non-zero vector $x$ of $\calL$ that lies in the $s_{u,\calL}$-orthogonal complement of
$H$. In particular $s(x,u(x)) \neq 0$. We also note that $x \in (H \oplus u(H))^{\bot_s}$; moreover, since $s(x,u(x)) \neq 0$ the $s$-orthogonal projection of
$u(x)$ onto $(H \oplus u(H))^{\bot_s}$ is not a scalar multiple of $x$, yielding that the said symplectic automorphism maps $x$ to
a vector that is not a scalar multiple of $x$. This proves the claimed statement.
Hence, in applying the space-pullback technique to $H$ and $u$, we find that the resulting symplectic automorphism $[u]$
is the direct sum of a cyclic symplectic automorphism of rank $2$, whose characteristic polynomial we denote by $q$, and of the identity of a
$2$-dimensional space.

Next, by combining Lemmas \ref{lemma:adaptlemmairr} and \ref{lemma:adaptlemmasplit},
we can find an $s_{u,H}$-skew-selfadjoint $\calC(t^2+1)$-automorphism $v$.

Now, we split the proof into three cases.

\vskip 3mm
\noindent \textbf{Case 1: $|\F|>3$ and $q \neq (t+1)^2$.} \\
Then, by point (b) of Lemma \ref{lemma:cyclic2+typeVI}, the resulting symplectic automorphism $[u]$
is i-adjacent to a $3$-reflectional symplectic $\calC(q(t^2))$-automorphism
$v'$, and we note that $q(t^2)$ is relatively prime with $t^2+1$.
By applying the space-pullback technique, we deduce that  $u$ is i-adjacent to some $u' \in \Sp(s)$
which splits into $u'_1 \botoplus u'_2$, where:
\begin{itemize}
\item either $u'_1$ is $\calC((t^2+1)^2)$ or it has invariant factors $t^2+1$ and $t^2+1$;
\item $u'_2$ is symplectically similar to $v'$.
\end{itemize}
In any case, $u'_1$ is $3$-reflectional (by Lemma \ref{lemma:(t^2+1)2} in the first case, and by Nielsen's theorem in the second one).
Since $v'$ is $3$-reflectional, we conclude that $u'$ is $3$-reflectional.

\vskip 3mm
\noindent \textbf{Case 2: $q=(t+1)^2$ or $q=(t-1)^2$.} \\
We can choose a symplectic involution $i_1$ of $(H \oplus u(H))^{\bot_s}$ such that $i_1[u]$
is the direct sum of a $\calC((t+1)^2)$-automorphism and of the identity of a $2$-dimensional space.
Applying the space-pullback technique with the subspace $W$, the symplectic form $(x,y) \mapsto s_{u,W}(x,v^{-1}(y))$
and the residual involution $i_1$, by Lemma \ref{lemma:recog} we find that $u$ is i-adjacent to some $u' \in \Sp(s)$ which splits into
$u'_1 \botoplus u'_2$ where:
\begin{itemize}
\item either $u'_1$ is $\calC((t^2+1)^2)$ or its invariant factors are $t^2+1$ and $t^2+1$;
\item $u'_2$ is the direct sum of a $\calC((t+1)^2)$-cell and of the identity of a $2$-dimensional space;
\end{itemize}
Note that $u'_2$ is $3$-reflectional by Lemma \ref{lemma:(t+1)2+identity}.
If $u'_1$ has invariant factors $t^2+1$ and $t^2+1$, then it is $2$-reflectional by Nielsen's theorem, and hence $u'$ is $3$-reflectional.
If $u'_1$ is $\calC((t^2+1)^2)$, then Lemma \ref{lemma:from(t2+1)2tot^4+1} applied to $u'_2$ yields that
$u'$ is i-adjacent to the orthogonal direct sum $u''$ of two $\calC((t^2+1)^2)$-cells;
then Nielsen's theorem shows that $u''$ is $2$-reflectional in $\Sp(s)$. In any case we conclude that $u$ is $4$-reflectional in $\Sp(s)$.

\vskip 3mm
\noindent \textbf{Case 3: $|\F|=3$ and $q=t^2+1$.} \\
In that case, we see from Lemma \ref{lemma:cyclic2+typeVIbis} that $[u]$ is i-adjacent to a $\calC(t^4+1)$-automorphism $w$.
By applying the space-pullback technique once more, we find that $u$ is i-adjacent to some $u' \in \Sp(s)$ which splits into
$u'_1 \botoplus u'_2$ where:
\begin{itemize}
\item either $u'_1$ is $\calC((t^2+1)^2)$ or its invariant factors are $t^2+1$ and $t^2+1$;
\item $u'_2$ is $\calC(t^4+1)$.
\end{itemize}
By Wall's theorem, $u'_2$ is a symplectic extension of a $\calC(t^2-t-1)$-automorphism, and we deduce from Proposition
\ref{prop:sympextension3refl} that it is $3$-reflectional. Besides, if $u'_1$ is not indecomposable then it is $2$-reflectional and hence $u'$ is $3$-reflectional.

Assume finally that $u'_1$ is $\calC((t^2+1)^2)$. Then, by Lemma \ref{lemma:from(t2+1)2tot^4+1}, we find that
$u'_1$ is i-adjacent to a $\calC(t^4+1)$-symplectic transformation. And we deduce that $u'$ is i-adjacent to $v'_1 \botoplus u'_2$,
which is $2$-reflectional by Nielsen's theorem.

In any case we have shown that $u$ is $4$-reflectional.
\end{proof}

\subsection{Conclusion}

We are finally ready to complete the proof of point (a) in Theorem \ref{theo:main}.
So, let us start from an s-pair $(s,u)$ whose dimension $n$ is a multiple of $4$,
discarding the special case where $n=4$ and $|\F|=3$.
Let us consider a regular-exceptional decomposition
$$u=u_r \botoplus u_e$$
of it, with the respective ranks of $u_r$ and $u_e$ denote by $n_r$ and $n_e$.
Note that $u_e$ is $2$-reflectional by Nielsen's theorem.

\vskip 3mm
\noindent
\textbf{Case 1: $n_r$ is a multiple of $4$.} \\
If $n_r=0$, then $u$ is $2$-reflectional by Nielsen's theorem.

Assume now that $n_r>0$.
In most cases, we already know that $u_r$ is $4$-reflectional:
\begin{itemize}
\item If $n_r>4$ or $|\F|>3$, then this is given by Proposition \ref{prop:nrmultiplede4};
\item If $u_r$ is an indecomposable cell of type I, II, III or V, then this is given by
Lemmas \ref{lemma:typeI-III-V} and \ref{lemma:typeII}.
\item If $u_r$ is a symplectic extension of a $\calC(t^2-t-1)$-automorphism, then
Proposition \ref{prop:sympextension3refl} shows that it is $3$-reflectional.
\end{itemize}
As $u_e$ is $2$-reflectional, we obtain that in all those cases $u$ is $4$-reflectional.

Let us consider the remaining case. Then $|\F|=3$, $n_r=4$ and $u_r=v_1 \botoplus v_2$ where $v_1$ and $v_2$ are cyclic automorphisms of $2$-dimensional vector spaces.
Then $n_e>0$, and since $n_e$ is a multiple of $4$ we gather that $u_e$ is the sum of several indecomposables cells of type VI. Hence
we can split $u_e=w_1 \botoplus w_2 \botoplus w_3$ where $w_1$ and $w_2$ are cells of type VI and $w_3$ is an orthogonal direct sum of cells of type VI.
Hence $u = (v_1 \botoplus w_1) \botoplus (v_2 \botoplus w_2) \botoplus w_3$, and by Proposition \ref{prop:rank2+typeVI}
each one of $v_1 \botoplus w_1$ and $v_2 \botoplus w_2$ is $4$-reflectional. As $w_3$ is $2$-reflectional, we conclude that $u$ is $4$-reflectional.

\vskip 3mm
\noindent
\textbf{Case 2: $n_r$ is not a multiple of $4$.} \\
We can split $u_e=w_1 \botoplus w_2$ where $w_1$ is an indecomposable cell of type VI, and $w_2$ is an orthogonal direct sum of such cells.
Hence, it will suffice to see that $u_r \botoplus w_1$ is $4$-reflectional.
If $n_r=2$ this follows directly from Proposition \ref{prop:rank2+typeVI}.
Assume now that $n_r \geq 6$. If $n_r=6$ and $n_e=2$, the result follows from Lemma \ref{lemma:rkur=6}.
Otherwise, $n_r+n_e \geq 12$ and the result follows directly from Lemma \ref{lemma:nr+nesup12}.

Hence, the proof of point (a) in Theorem \ref{theo:main} is now completed.

\section{A complete analysis of $\Sp_4(\F_3)$}\label{section:F3}

In this section, we take $\F=\F_3$, we fix a $4$-dimensional symplectic space $(V,s)$ over $\F_3$ and
we describe, for each element $u$ in $\Sp(s)$, the least integer $k$ such that $u$ is $k$-reflectional, which we will call the \textbf{reflectional length} of $u$
and denote by $\rl(u)$. Note that if $u$ has reflectional length $k>0$, then so does $-u$ unless $u=-\id$.

One of our aims is the following result, but we will go much further than Theorem \ref{theo:main}, since we will determine the reflectional length of every element of $\Sp(s)$
in terms of its conjugacy class (see Section \ref{section:table} for a table of the full results).

\begin{theo}\label{theo:F3}
Let $s$ be a $4$-dimensional symplectic form over $\F_3$. Then every element of $\Sp(s)$ is $5$-reflectional.
\end{theo}

Of course, the elements of reflectional length $1$ are $-\id$ and the non-trivial involutions.
By Nielsen's theorem, the elements of reflectional length $2$ are:
\begin{itemize}
\item The quarter turns, i.e.\ the $u \in \Sp(s)$ such that $u^2=-\id$;
\item The $u \in \Sp(s)$ with minimal polynomial $(t-\eta)^2$ and with hyperbolic Wall invariant (attached to $(t-\eta,2)$), for $\eta=\pm 1$.
\end{itemize}

Before we examine the remaining elements, let us recall some useful results from \cite{dSPinvol4Symp} (see section 7 therein).

\begin{lemma}[Trace Trick]\label{lemma:trace}
Let $u \in \Sp(s)$ be $3$-reflectional but not $2$-reflectional.
Then $\tr u=0$.
\end{lemma}

\begin{lemma}[Commutation Lemma]\label{lemma:commute}
Let $a$ and $b$ be automorphisms of a vector space that are annihilated by monic polynomials $p$ and $q$
of degree $2$ such that $p(0)q(0) \neq 0$. Then, for $u:=ab$, the elements $a$ and $b$
commute with $u+p(0)q(0)u^{-1}$.
\end{lemma}

\begin{lemma}[Reformulation of lemma 7.7 in \cite{dSPinvol4Symp}]\label{lemma:F3(t2-1)de3a4}
Let $u \in \Sp(s)$ have minimal polynomial $(t^2-1)^2$, with Wall invariants such that $(s,u)_{t-1,2} \simeq -(s,u)_{t+1,2}$.
If $u$ is not $3$-reflectional in $\Sp(s)$ then it is not $4$-reflectional either.
\end{lemma}

\begin{lemma}[Proposition 7.6 from \cite{dSPinvol4Symp}]\label{lemma:F3rappel}
Let $u \in \Sp(s)$ have minimal polynomial $(t^2+1)(t-\eta)^2$ for some $\eta=\pm 1$.
Then $u$ is not $4$-reflectional in $\Sp(s)$.
\end{lemma}

\subsection{Elements with reflectional length $3$}

\begin{lemma}\label{lemma:F3(t^2-1)^2pas3}
Let $u \in \Sp(s)$ have minimal polynomial $(t^2-1)^2$
and be such that $(s,u)_{t-1,2} \simeq -(s,u)_{t+1,2}$.
Then $u$ is not $3$-reflectional in $\Sp(s)$.
\end{lemma}

\begin{proof}
Throughout the proof, we consider the quadratic form
$$Q_u : x \in V \mapsto s(x,u(x)).$$
The symmetric bilinear form that is attached to $Q_u$ is
$$(x,y) \in V^2 \mapsto \frac{1}{2} s(x,(u-u^{-1})(y))=-s(x,(u-u^{-1})(y)),$$
and hence its regular part is equivalent to $-(s,u)_{t-1,2} \bot -(s,u)_{t+1,2}$.
It follows from our assumptions that the regular part of $Q_u$ is hyperbolic.

Assume now that there is an involution $i \in \Sp(s)$ such that $v:=iu$ is $2$-reflectional.
We note that $u$ has only one Jordan cell of size $2$ for the eigenvalue $1$, and hence $u$ is not $2$-reflectional in $\Sp(s)$.
Hence $v$ is not an involution; by Nielsen's theorem $v$ is a symplectic extension of an endomorphism $w$ of a $2$-dimensional vector space
which is similar to its inverse, and $w$ cannot be an involution whence the caracteristic polynomial of $w$ equals $t^2+\alpha t+1$
for some scalar $\alpha$; and we deduce from the Cayley-Hamilton theorem that $t^2+\alpha t+1$ annihilates $v$.
The Commutation Lemma then shows that both $i$ and $v$ commute with $u-u^{-1}$.

Hence $i$, $v$ and $u$ induce automorphisms of the quotient space $\overline{V}:=V/\Ker(u-u^{-1})$, which has dimension $2$.
We denote these authomorphisms by $\overline{i}$, $\overline{v}$ and $\overline{u}$, respectively. And we even obtain that
$\overline{i}$, $\overline{v}$ and $\overline{u}$ are orthogonal automorphisms of $\overline{V}$ for the quadratic form induced on $\overline{V}$ by
$Q_u$, which is hyperbolic.

Note that $\overline{u}$ is a non-trivial involution, and that $\overline{i}$ is an involution. If $\overline{i}$ were trivial, then
by Remark \ref{remark:pseudoextension} $i$ would have only one eigenvalue and hence it would be trivial, and then $u=iv=\pm v$ would be $2$-reflectional, which is not true. Hence $\overline{i}$ is a non-trivial involution of $\overline{V}$.

Yet, in a hyperbolic plane over $\F_3$, there are only two non-isotropic lines, and hence there are only
two non-trivial symmetries. Thus $\overline{i}=\pm \overline{u}$, leading to $\overline{v}=\pm \id$.
Multiplying $v$ with $-1$ if necessary, we can assume that $\overline{v}=\id$.
In that case, we see from Remark \ref{remark:pseudoextension} that
$\chi_v=(t-1)^4$, while $v(x)=x$ for all $x \in \Ker(u-u^{-1})$, and we deduce that the minimal polynomial of $v$ is $(t-1)^2$ and that $v$ has two Jordan cells of size $2$ for the eigenvalue $1$ (again, by Nielsen's theorem). Thus $\im(v-\id) \subseteq \Ker(u-u^{-1})$ and $\im(v-\id)$ has dimension $2$, leading to $\im(v-\id)=\Ker(u-u^{-1})$.

We will finish with a matrix computation. We can take a symplectic basis $\bfB=(e_1,e_2,f_1,f_2)$ in which the first two vectors
constitute a basis of $\Ker(u-u^{-1})$ with $i(e_1)=e_1$ and $i(e_2)=-e_2$, and $i(f_1)=f_1$ and $i(f_2)=-f_2$.
From $\overline{v}=\id$, we get that the respective matrices of $i$ and $v$ in $\bfB$ are
$$D_2:=\begin{bmatrix}
D & 0 \\
0 & D
\end{bmatrix} \quad \text{and} \quad
V=\begin{bmatrix}
I_2 & \Delta \\
0 & I_2
\end{bmatrix}$$
where $D=\begin{bmatrix}
1 & 0 \\
0 & -1
\end{bmatrix}$ and $\Delta \in \Mat_2(\F)$. Since $V \in \Sp_4(\F)$, the matrix $\Delta$ is symmetric.

Note that $-\Delta$ represents the regular part of the quadratic form $Q_v$, so it must be hyperbolic, and hence $\det(\Delta)=\det(-\Delta)=-1$.

Next, $u$ is represented by
$U:=\begin{bmatrix}
D & D\Delta \\
0 & D
\end{bmatrix}$ in a symplectic basis. In this basis, the Gram matrix of $Q_u$ is then equal to
$$\frac{1}{2} K_4 (U-U^{-1})=\frac{1}{2}\,K_4(D_2 V-V^{-1} D_2^{-1})=\frac{1}{2}\,K_4\,\begin{bmatrix}
0_2 & D\Delta+\Delta D \\
0_2 & 0_2
\end{bmatrix}=\begin{bmatrix}
0_2 & 0_2 \\
0_2 & -\frac{D \Delta+(D \Delta)^T}{2} \\
\end{bmatrix}.$$
Hence, the regular part of $Q_u$, which has rank $2$, has a Gram matrix that is the opposite of the symmetric part of $D \Delta$.

Writing
$\Delta=\begin{bmatrix}
a & b \\
b & c
\end{bmatrix}$, we find $D\Delta=\begin{bmatrix}
a & b \\
-b & -c
\end{bmatrix}$ and hence the opposite of the symmetric part of $D\Delta$ has determinant $-ac$. As the regular part of $Q_u$ is hyperbolic,
we deduce that $-ac=-1$. Yet $ac-b^2=\det \Delta=-1$, leading to $-ac=1-b^2 \in \{0,1\}$.
This contradiction shows that $u$ is not $3$-reflectional in $\Sp(s)$.
\end{proof}

We turn to the difficult case of cyclic automorphisms with minimal polynomial $(t^2+1)^2$.
The following lemma will be useful:

\begin{lemma}\label{lemma:involutionLagrangian}
Let $i$ be a symplectic involution in $\Sp(s)$, and
$\calL$ be a Lagrangian of $(V,s)$.
Then $\dim \calL-\dim (\calL \cap i(\calL))$ is even.
\end{lemma}

\begin{proof}
Note first that $i^\star=i^{-1}=i$, whence $i$ is $s$-selfadjoint, and since $s$ is symplectic this yields that
$b : (x,y) \in V^2 \mapsto s(x,i(y))$ is alternating. Yet
the left-radical of the bilinear form $b_\calL : (x,y) \in \calL^2 \mapsto b(x,y)$ equals $\calL \cap i(\calL)^{\bot_s}=\calL \cap i(\calL^{\bot_s})=\calL \cap i(\calL)$ and hence
the rank of $b_\calL$ equals $\dim \calL-\dim(\calL \cap i(\calL))$. Since $b_\calL$ is alternating, this rank is even,
and the proof is complete.
\end{proof}

\begin{lemma}\label{lemma:(t2+1)^2F3}
Let $u \in \Sp(s)$ have minimal polynomial $(t^2+1)^2$. Then $u$ is not $3$-reflectional in $\Sp(s)$.
\end{lemma}

This result shows that the provision $|\F|>3$ in Lemma \ref{lemma:(t^2+1)2} was unavoidable.

\begin{proof}
We already know that $u$ is not $2$-reflectional. Assume that $u=i v$ where $i$ is an involution and $v$
is $2$-reflectional. Then $v$ is not $1$-reflectional, and just like in the proof of Lemma \ref{lemma:F3(t^2-1)^2pas3}
we deduce from Nielsen's theorem that $v^2+\alpha v+\id=0$ for some scalar $\alpha$.
Besides $i$ is nontrivial.

\noindent \textbf{Case 1:} $(v-\eta \id)^2=0$ for some $\eta=\pm 1$. \\
Note that $\calL:=\Ker(u^2+\id)=\im (u^2+\id)=\Ker(u^2+\id)^{\bot_s}$ is a Lagrangian of $V$
and $u$ induces a quarter turn of $\calL$.
If $\calL$ were stable under $i$ then it would also be stable under $v$, the endomorphism induced by $v$ would have determinant $1$
(being triangularizable with a sole eigenvalue, equal to $\pm 1$), the involution induced by $i$ would have determinant $1$
and hence it would be trivial; and finally $i=\pm \id$ because $i$ is $s$-symplectic.

So, $i(\calL) \neq \calL$, and by Lemma \ref{lemma:involutionLagrangian} this yields $i(\calL) \cap \calL=\{0\}$.
It follows that $u$ is obtained from $v$ by space-pullbacking with the Lagrangian $\calL$ and the involution $i$.
This leads us to investigate the bilinear form $s_{v,\calL}$. Noting that $v+v^{-1}=2\eta\id$, we
see that the skew-symmetric part of $s_v$ vanishes on every $s$-Lagrangian and in particular on $\calL$.
Hence $s_{v,\calL}$ is symmetric. Moreover $2s_{v,\calL}$ is a restriction of $(x,y) \mapsto s(v,(v-v^{-1})(y))$,
which has rank $2$ and with regular part isometric to the quadratic Wall invariant $(s,v)_{t-\eta,2}$.
Since $s_{v,\calL}$ is nondegenerate, we deduce that it is isometric to $\frac{1}{2}(s,v)_{t-\eta,2}$, which is hyperbolic because
$v$ is $2$-reflectional. Hence $s_{v,\calL}$ is hyperbolic. We can finally choose $x \in \calL \setminus \{0\}$ to be
$s_{v,\calL}$-isotropic. Then $s_{v,\calL}(x,u^{-1}(x))=s(x,i(x))=0$ because $i$ is $s$-selfadjoint.
It follows that $u^{-1}(x) \in \F x$, which contradicts the fact that $u$ has no eigenvalue.

\vskip 3mm
\noindent \textbf{Case 2:} $v^2=-\id$. \\
By Lemma \ref{lemma:commute}, it follows that $i$ and $v$ commute with $w:=u-u^{-1}$. Next, note that $w^2=(u+u^{-1})^2-4 \id=-\id$.
Now, we can naturally endow $V$ with a structure of vector space (of dimension $2$) over the field $\L:=\F[w]$ (which is isomorphic to $\F_9$).
The $s$-adjunction operator leaves $\L$ invariant since $w^\star=u^\star-u=-w=w^{-1}$, and it is the non-identity automorphism of the field $\L$.
The $\F$-bilinear form $(x,y) \in V^2 \mapsto s(x,w(y))$ is symmetric and nondegenerate.
For all $(x,y)\in V^2$, there is a unique $h(x,y) \in \L$ such that
$$\forall \lambda \in \L, \quad s(x,w(\lambda y))=\Tr_{\L/\F}(\lambda h(x,y)).$$
Obviously, $h$ is right-$\L$-linear and nondegenerate.
Moreover, for all $(x,y)\in V^2$ and all $\lambda \in \L$, we have
\begin{multline*}
s(y,w(\lambda x))=s(-w(y),\lambda x)=s(-\lambda^\star w(y),x)=s(x,w(\lambda^\star y)) \\
=\Tr_{\L/\F}(\lambda^\star h(x,y))=\Tr_{\L/\F}(\lambda h(x,y)^\star)
\end{multline*}
and hence $h(y,x)=h(x,y)^\star$. Hence $h$ is a Hermitian form on the $\L$-vector space $V$.

Next, since $i$ and $v$ commute with $w$ and since they are isometries for $s$, we gather that they belong to the unitary group $\calU(h)$ of $h$.
And now, $h$ is hyperbolic, as is any Hermitian form of even rank over a finite field. So, it has as many
isotropic lines as there are elements of norm $1$ in $\L$. Hence there are four such lines, and we denote their set by $X$.
The unitary group $\calU(h)$ then naturally acts on $X$. Next, $i$ is a non-trivial involution in $\calU(h)$ so its eigenspaces are non-isotropic lines, and hence
$i$ leaves no element of $X$ invariant. But $i^2=\id$ so $i$ induces a double-transposition of $X$.

Next, note that $(wv)^2=w^2v^2=\id$. Hence either $v=\pm w^{-1}$, in which case $v$ induces the identity of $X$,
or $v$ has two distinct eigenvalues in $\L$ (namely $w^{-1}$ and $-w^{-1}$), and the corresponding eigenspaces are $h$-orthogonal to one another; in the latter case,
we obtain, as we did for $i$, that $v$ induces a double-transposition of $X$.
In any case, we conclude that $u$ induces a permutation of $X$ that belongs to the Klein subgroup of $\mathfrak{S}(X)$,
i.e.\ it is either the identity or a double-transposition.

Note that the elements of $X$ are simply the Lagrangians for $s$ that are stable under $w$. In particular,
the Lagrangian $\calL:=\Ker(u^2+\id)$ belongs to $X$, and since $u(\calL)=\calL$ we deduce that $u$ induces the identity of $X$.
Hence, by taking $\calL' \in X \setminus \{\calL\}$, we find that $\calL'$ is a Lagrangian that is transverse to $\calL$
(because $\calL$ and $\calL'$ have trivial intersection as $\L$-linear subspaces) and stable under $u$.
And then $u=s_{\calL'}(u_{\calL})$ would be annihilated by $t^2+1$, which is false.
This final contradiction completes the proof.
\end{proof}

We are now ready to obtain the full classification of $3$-reflectional elements in $\Sp_4(\F_3)$.

\begin{prop}\label{prop:F3refllength3}
The elements of $\Sp(s)$ with reflectional length $3$ are:
\begin{itemize}
\item The elements with minimal polynomial $(t-\eta)(t+\eta)^2$ for some $\eta=\pm 1$.
\item The elements with minimal polynomial $(t^2-1)^2$ and equivalent Wall invariants attached to $(t-1,2)$ and $(t+1,2)$;
\item The elements with minimal polynomial $t^4+1$.
\end{itemize}
\end{prop}

\begin{proof}
We have seen in Lemma \ref{lemma:(t+1)2+identity}
that every element of $\Sp(s)$ with minimal polynomial $(t-1)(t+1)^2$ has reflectional length at most $3$.
And Nielsen's theorem shows that such an element cannot be $2$-reflectional.
Multiplying with $-1$, we deduce that every every element of $\Sp(s)$ with minimal polynomial $(t+1)(t-1)^2$ has reflectional length $3$.

Next, let $u \in \Sp(s)$ have minimal polynomial $(t^2-1)^2$ and
equivalent Wall invariants attached to $(t-1,2)$ and $(t+1,2)$.
Then $u=u_1 \botoplus u_2$ where $u_1$ and $u_2$ are respectively $\calC((t-1)^2)$ and $\calC((t+1)^2)$ symplectic transformations.
The quadratic Wall invariants of $u$ are non-hyperbolic, and hence $\rl(u) >2$ by Nielsen's theorem, yet $u$ is i-adjacent to $u':=u_1 \botoplus (-u_2)$,
for which the sole non-zero Jordan number is $n_{t-1,2}(u')$, which equals $2$, and the sole quadratic Wall invariant
is $(s,u')_{t-1,2} \simeq (s,u)_{t-1,2} \bot \left(-(s,u)_{t+1,2}\right)$, which is hyperbolic. Hence $u'$ is $2$-reflectional,
and we deduce that $\rl(u)=3$.

Finally, if $u$ has minimal polynomial $t^4+1$, then it is a symplectic extension of a $\calC(t^2-t-1)$-automorphism,
and hence $u$ is $3$-reflectional by Proposition \ref{prop:sympextension3refl}.
Once more, we already know from Nielsen's theorem that $u$ is not $2$-reflectional in $\Sp(s)$.

Conversely, let $u \in \Sp(s)$ be $3$-reflectional but not $2$-reflectional. By the Trace Trick, $\tr(u)=0$.
Hence the characteristic polynomial of $u$, which is a palindromial $q$ such that $q(0)=1$, equals $t^4+at^2+1$ for some $a \in \F_3$.
If $a=0$ then $u$ is of the third type cited in the proposition.
If $a=-1$ then either $u$ has minimal polynomial $(t^2+1)^2$, in which case Lemma \ref{lemma:(t2+1)^2F3} contradicts our assumptions, or
it has minimal polynomial $t^2+1$, in which case it is $2$-reflectional.

Assume finally that $a=1$, so that the characteristic polynomial of $u$ equals $(t^2-1)^2$. As $u$ is not $2$-reflectional,
the minimal polynomial of $u$ cannot be $t^2-1$, and hence it equals either $(t-1)(t+1)^2$ or $(t+1)(t-1)^2$ or $(t^2-1)^2$.
Assume finally that the minimal polynomial of $u$ is $(t^2-1)^2$.
Then by Lemma \ref{lemma:F3(t^2-1)^2pas3} we cannot have
$(s,u)_{t-1,2} \simeq -(s,u)_{t+1,2}$, and hence $(s,u)_{t-1,2} \simeq (s,u)_{t+1,2}$ because
the bilinear forms $(s,u)_{t-1,2}$ and $(s,u)_{t+1,2}$ are symmetric with rank $1$.
Thus $u$ is of the second stated type, which completes the proof.
\end{proof}

\subsection{Elements with reflectional length $4$}

Next, we determine the elements with reflectional length $4$.
From Lemmas \ref{lemma:typeI-III-V} and \ref{lemma:typeII}, we deduce that if $u \in \Sp(s)$
is indecomposable, then it is $4$-reflectional unless $u$ is $\calC(t^4+1)$.
But we have seen earlier that an element with reflectional length $3$ is indecomposable if and only if its minimal polynomial is $t^4+1$.
Hence, the indecomposable elements with reflectional length $4$ are exactly the indecomposable elements with minimal polynomial different from $t^4+1$.

Next, we consider the decomposable elements.
By Proposition \ref{prop:rank2+typeVI}, we know that every element of the form $v \botoplus \eta \id$, with $\eta=\pm 1$
and $v$ a symplectic transformation of rank $2$, is $4$-reflectional. And we also know from Proposition \ref{prop:F3refllength3} which of those are $3$-reflectional.

If an element $u$ has minimal polynomial $(t^2-1)^2$, then either its sole nontrivial quadratic Wall invariants are
equivalent, in which case $u$ is $3$-reflectional, or they are inequivalent.
In the latter case, it has been recalled as Lemma \ref{lemma:F3(t2-1)de3a4} (see also lemma 7.7 of \cite{dSPinvol4Symp}) that if $u$ is not $3$-reflectional then it is not $4$-reflectional either; yet we have just seen in Lemma \ref{lemma:F3(t^2-1)^2pas3} that $u$ is not $3$-reflectional.
We conclude that an element with minimal polynomial $(t^2-1)^2$ cannot have reflectional length $4$.

This leaves us with two remaining cases to consider:
\begin{itemize}
\item The symplectic transformations with minimal polynomial $(t^2+1)(t-\eta)^2$ for some $\eta=\pm 1$;
\item The symplectic transformations with invariant factors $(t-\eta)^2,(t-\eta)^2$ for some $\eta=\pm 1$, and a non-hyperbolic Wall invariant.
\end{itemize}
It has been recalled as Proposition \ref{lemma:F3rappel} (see proposition 7.6 in \cite{dSPinvol4Symp}) that the former are not $4$-reflectional.
The latter are tackled in the next lemma:

\begin{lemma}\label{lemma:(t-1)^2not4}
Let $u \in \Sp(s)$ have invariant factors $(t-\eta)^2$ and $(t-\eta)^2$ for some $\eta=\pm 1$, and non-hyperbolic Wall invariant attached to $(t-\eta,2)$.
Then $u$ is $4$-reflectional in $\Sp(b)$.
\end{lemma}

\begin{proof}
Without loss of generality, we can assume that $\eta=1$.
Let $S \in \Mats_2(\F_3)$ be a symmetric matrix.
Consider the matrix
$$U_S:=\begin{bmatrix}
I_2 & S \\
0 & I_2
\end{bmatrix}.$$
If $S$ is invertible, we note that $U_S$ represents, in a symplectic basis of $(V,s)$, a symplectic transformation $v$ with two invariant factors, all
equal to $(t-1)^2$, and $-S$ as a Gram matrix for the Wall invariant $(s,v)_{t-1,2}$.
In particular, if $S$ is hyperbolic then $v$ is $2$-reflectional in $\Sp(s)$;
and if $S$ is non-hyperbolic then $U_S$ represents $u$ in a well-chosen symplectic basis of $(V,s)$.
Noting that $U_{S+T}=U_S U_T$ for all $S,T$ in $\Mats_2(\F)$, it will suffice to find two hyperbolic matrices $S,T$ in $\Mats_2(\F_3)$
such that $S+T$ is invertible yet non-hyperbolic (for then both $U_S$ and $U_T$ are $2$-reflectional in $\Sp_4(\F_3)$, and hence
$U_{S+T}$ is $4$-reflectional in $\Sp_4(\F_3)$).
We conclude by noting that $S:=\begin{bmatrix}
1 & 0 \\
0 & -1
\end{bmatrix}$ and $T:=\begin{bmatrix}
0 & 1 \\
1 & 0
\end{bmatrix}$ are hyperbolic but their sum $\begin{bmatrix}
1 & 1 \\
1 & -1
\end{bmatrix}$ has determinant $1$ and hence is invertible but non-hyperbolic.
\end{proof}

\subsection{Elements with reflectional length $5$}\label{}

We are left with two last types of elements, and we will prove that they have reflectional length $5$.

\begin{prop}
Let $u \in \Sp(s)$ have minimal polynomial $(t^2+1)(t-\eta)^2$ for some $\eta=\pm 1$.
Then $u$ has reflectional length $5$.
\end{prop}

\begin{proof}
For some symplectic basis $(e_1,e_2,f_1,f_2)$ of $V$ and some $\varepsilon=\pm 1$,
the matrix of $u$ in $(e_1,f_1,e_2,f_2)$ equals
$K \oplus T$ where $K:=\begin{bmatrix}
0 & -1 \\
1 & 0
\end{bmatrix}$ and
$T:=\begin{bmatrix}
\eta & \varepsilon \\
0 & \eta
\end{bmatrix}$.
Replacing $u$ with $u^{-1}$ if necessary, we can assume that $\varepsilon=1$.
Then one sees that the endomorphism $i$ of $V$ whose matrix in  $(e_1,f_1,e_2,f_2)$ equals
$\begin{bmatrix}
0_2 & I_2 \\
I_2 & 0_2
\end{bmatrix}$ is an $s$-symplectic involution, and then as the matrix of $iu$ in the same basis
equals
$\begin{bmatrix}
0_2 & T \\
K & 0_2
\end{bmatrix}$ we gather that its characteristic polynomial equals $t^4-\tr(KT)\,t^2+1=t^4-t^2+1=(t^2+1)^2$.
Hence $iu$ is $4$-reflectional in $\Sp(s)$, and we conclude that $u$ is $5$-reflectional in $\Sp(s)$.
Besides, by Proposition \ref{lemma:F3rappel} we know that $u$ is not $4$-reflectional in $\Sp(s)$.
\end{proof}

\begin{prop}
Let $u \in \Sp(s)$ have minimal polynomial $(t^2-1)^2$, with
$(s,u)_{t-1,2} \not\simeq (s,u)_{t+1,2}$. Then $u$ has reflectional length $5$.
\end{prop}

\begin{proof}
It has been seen in the previous section that $u$ is not $4$-reflectional. Hence it will suffice to prove that it is $5$-reflectional.
The subspace $\calL:=\Ker(u^2-\id)$ is an $s$-Lagrangian.
The restriction $u_{\calL}$ is $\calC(t^2-1)$. We can choose a Lagrangian $\calL'$ that is transverse to $\calL$,
and by Lemma \ref{lemma:cyclicfitplain} we can choose an involution $i \in \GL(\calL)$ such that $iu_\calL$ is $\calC(t^2+1)$.
Then $s_{\calL'}(i) u$ has characteristic polynomial $(t^2+1)^2$, and hence it is $4$-reflectional in $\Sp(s)$.
As $s_{\calL'}(i)$ is a symplectic involution of $\Sp(s)$, we conclude that $u$ is $5$-reflectional in $\Sp(s)$.
\end{proof}

In particular, we have proved that every element of $\Sp(s)$ is $5$-reflectional, thereby completing the proof of Theorem
\ref{theo:F3}.

\subsection{Conclusion}\label{section:table}

We gather the previous results in the following table, leaving aside the elements with reflectional length $0$ or $1$.
In this table, $s$ is a symplectic form of rank $4$ over $\F_3$. It is useful to note that the sole monic irreducible palindromials
of degree $4$ over $\F_3$ are $t^4+t^3+t^2+t+1$ and $t^4-t^3+t^2-t+1$.

\begin{table}[H]
\caption{The reflectional length of any $u \in \Sp(s)$ that is not an involution}
\label{figure2}
\begin{center}
\begin{tabular}{| c | c | c |}
\hline
Invariant factors & Constraint on Wall invariants & Reflectional length  \\
\hline
\hline
$t^2+1,t^2+1$ &  & $2$ \\
\hline
$(t-\eta)^2,(t-\eta)^2$ & $(s,u)_{t-\eta,2}$ hyperbolic & $2$ \\
for some $\eta=\pm 1$ & & \\
\hline
$t^4+1$ &  & 3\\
\hline
$(t+\eta)^2(t-\eta),t-\eta$ &  & $3$\\
for some $\eta=\pm 1$ & & \\
\hline
$(t^2-1)^2$ & $(s,u)_{t-1,2} \simeq (s,u)_{t+1,2}$ & $3$ \\
\hline
$t^4+t^3+t^2+t+1$ &  & $4$ \\
\hline
$t^4-t^3+t^2-t+1$ &  & $4$ \\
\hline
$(t-\eta)^4$ &  & $4$ \\
for some $\eta=\pm 1$ & & \\
\hline
$(t^2+1)^2$ &  & $4$ \\
\hline
$(t-\eta)^2,(t-\eta)^2$ & $(s,u)_{t-\eta,2}$ non-hyperbolic & $4$ \\
for some $\eta=\pm 1$ & &  \\
\hline
$(t^2-1)^2$ & $(s,u)_{t-1,2} \simeq -(s,u)_{t+1,2}$ & $5$ \\
\hline
$(t^2+1)(t-\eta)^2$ &  & $5$ \\
for some $\eta=\pm 1$ & & \\
\hline
\end{tabular}
 \end{center}
\end{table}

\section{Dimensions that are not multiples of $4$}\label{section:dim2mod4}

In this section, we complete the proof of Theorem \ref{theo:main} by proving point (b).
The proof relies on the solution of the $6$-dimensional case, which was almost entirely solved in \cite{dSPinvol4Symp},
and on point (a) and (c) of Theorem \ref{theo:main}, which solve the case of dimensions that are multiples of $4$.
One of the main ideas indeed is to reduce the situation to the case of dimensions that are multiples of $4$,
thanks to what we call the \emph{plane fixing technique}.

The plane fixing technique works as follows. By ``plane", we mean a $2$-dimensional linear subspace.
Let $(s,u)$ be an s-pair with dimension $n \geq 10$ which is not a multiple of $4$.
Assume that $u$ is i-adjacent to some $u' \in \Sp(s)$ that fixes every vector of some \emph{$s$-regular} plane $P$. Then we split $u'=\id_P \botoplus v$,
where the underlying vector space of $v$ is $P^{\bot_s}$, which has dimension $n-2$. As $n-2$ is a multiple of $4$ and is greater than $4$,
point (a) of Theorem \ref{theo:main} yields that $v$ is $4$-reflectional in the corresponding symplectic group, and hence so is
$u'$. And it follows that $u$ is $5$-reflectional in $\Sp(s)$.

It remains to see how to find $u'$ in practice. We will search for an $s$-regular plane $P$ such that $P \bot_s u(P)$, and
we will say that such a plane is \textbf{adapted} to $u$.
Then we can construct $i$ as the linear involution of the underlying vector space $V$ of $(s,u)$ such that
$i(x)=u(x)$ for all $x \in P$, $i(x)=u^{-1}(x)$ for all $x \in u(P)$, and $i(x)=x$ for all $x \in (P \oplus u(P))^{\bot_s}$.
It is easily checked that $i$ is an involution and that it belongs to $\Sp(s)$. And from the definition we readily obtain that
$(iu)(x)=x$ for all $x \in P$.
In practice, finding such planes $P$ is difficult, but we will manage to do it in most situations of indecomposable cells.

The remainder of this section is laid out as follows: in Section \ref{section:dim6}, we recall
the results that we proved in \cite{dSPinvol4Symp}, and we complete the $6$-dimensional case.
In Section \ref{section:planes}, we obtain the existence of an adapted plane
for most indecomposable cells of type I to IV. In Section \ref{section:typeV}, we tackle indecomposable cells of type V:
there, we do not obtain an adapted plane, but we still manage to find an involution $i \in \Sp(s)$
such that $iu$ fixes every vector of an $s$-regular plane.
Finally, in Section \ref{section:final}, we wrap the proof up, which requires that we care
about a special case for fields with more than $3$ elements (related to some specific indecomposable cells for which the results of Section \ref{section:planes} cannot be adapted).

\subsection{The $6$-dimensional case}\label{section:dim6}

This case has been almost entirely completed in \cite{dSPinvol4Symp}, even for finite fields, so we will be very brief.
In \cite{dSPinvol4Symp}, we have proved the following result, which holds over an arbitrary field (with characteristic not $2$):

\begin{prop}[Corollary 6.4 of \cite{dSPinvol4Symp}]\label{prop:dim6}
Let $(s,u)$ be an s-pair of dimension $6$ with no Jordan cell of odd size attached to $1$ or $-1$.
Then $u$ is $5$-reflectional in $\Sp(s)$.
\end{prop}

From there, completing the $6$-dimensional case is easy.
Let $(s,u)$ be an s-pair of dimension $6$. Thanks to Proposition \ref{prop:dim6}, it suffices to deal with the case where $u$ has Jordan cells of odd size attached to $1$ or $-1$, and now we assume that this case holds.
By the classification of indecomposable cells, it follows that there are two possibilities:
\begin{itemize}
\item Either $(s,u)$ is an indecomposable cell of type VI; in that case $u$ is $2$-reflectional in $\Sp(s)$;
\item Or $u=\eta \id_P \botoplus v$ for some plane $P$, some $\eta =\pm 1$ and some symplectic transformation $v$ of $P^{\bot_s}$;
in that case points (a) and (c) of Theorem \ref{theo:main} yield that $v$ is $5$-reflectional in the corresponding symplectic group,
and hence $u$ is $5$-reflectional in $\Sp(s)$.
\end{itemize}
Hence, in any case $u$ is $5$-reflectional in $\Sp(s)$.

\subsection{Finding adapted planes}\label{section:planes}

\begin{lemma}\label{lemma:typeItoIVfixplanen=1}
Assume that $\F$ is finite.
Let $(s,u)$ be an s-pair. Assume that $u$ is $\calC(p)$ for some polynomial $p$
of degree $d \geq 6$ which is either an irreducible palindromial or the product $qq^\sharp$ for some $q \in \Irr(\F)$ such that $q^\sharp \neq q$.
Then there exists an $s$-regular plane $P$ such that $u(P) \bot_s P$.
\end{lemma}

\begin{proof}
Note in any case that $p(0)=1$.

We use an algebraic model for $(s,u)$.
We consider the quotient $\F$-algebra $R:=\F[t]/(p)$, and we equip it with the non-identity involution
$x \mapsto x^\bullet$ that takes the coset $\lambda$ of $t$ to $\lambda^{-1}$. We write $\tr$ short for $\Tr_{R/\F}$, and
we say that an element $x \in R$ is Hermitian if $x^\bullet=x$, and skew-Hermitian if $x^\bullet =-x$.
We also set
$$S:=\{x \in R : \; x^\bullet=-x\}.$$

It is crucial to stress that either $R$ is a separable field extension of $\F$ or it is isomorphic to the direct product of two such extensions.
Hence, in any case $(x,y)\in R^2 \mapsto \tr(xy)$ is a nondegenerate symmetric bilinear form on the $\F$-vector space $R$.
Moreover, even if $R$ is not a field, every nonzero skew-Hermitian element is invertible.
Assume indeed that we have a non-invertible skew-Hermitian element $x \in R$ and $p=qq^\sharp$ for some $q \in \Irr(\F)$ such that $q^\sharp \neq q$. Then $r(\lambda)\,x=0$ for some $r \in \{q,q^\sharp\}$, and hence by applying the involution we also obtain
$-r(\lambda^{-1})\,x=0$, yielding $r^\sharp(\lambda)\,x=0$. By B\'ezout's theorem, this yields $x=0$.

Next, we choose an arbitrary non-zero skew-Hermitian element $h$ of $R$, to be adjusted
later, and we take the symplectic form
$$s' : (x,y) \in R^2 \mapsto \tr(hx^\bullet y)$$
and the automorphism $u' : x \mapsto \lambda x$ of the $\F$-vector space $R$, which is a $\calC(p)$-transformation in $\Sp(s')$.
As $\F$ is finite, Wall's theorem shows that $(s,u)$ is isometric to $(s',u')$, and hence no generality is lost in assuming that $(s,u)=(s',u')$, which we will now do.

We try to adjust $h$ so that $P:=\Vect(1,\lambda^2)$ is $s$-regular and $s$-orthogonal to $u(P)=\Vect(\lambda,\lambda^3)$.
The first condition is satisfied if and only if $\tr(h \lambda^2) \neq 0$, and the second one if and only if
$\tr(h \lambda)=\tr(h\lambda^{-1})=\tr(h\lambda^3)=0$, which reduces to $\tr(h \lambda)=\tr(h \lambda^3)=0$ because
$\tr(h \lambda^{-1})=\tr((h\lambda^{-1})^\bullet)=\tr(h^\bullet \lambda)=-\tr(h \lambda)$.

Now, assume that no $h \in S$ satisfies $\tr(h \lambda^2) \neq 0$ and $\tr(h \lambda)=\tr(h\lambda^3)=0$.
Then the kernel of the $\F$-linear map $f : h \in S \mapsto \tr(h \lambda^2)$ includes the intersection of those
of $f_1 : h \in S \mapsto \tr(h \lambda)$ and $f_2 : h \in S \mapsto \tr(h \lambda^3)$.
By duality theory, $f$ is a linear combination of $f_1$ and $f_2$, yielding $(\alpha,\beta)\in \F^2$ such that
$\lambda^2-\alpha \lambda-\beta \lambda^3$ is $\tr$-orthogonal to $S$. In turn, this means that
$\lambda^2-\alpha \lambda-\beta \lambda^3$ is Hermitian, yielding
$$\lambda^2-\alpha \lambda-\beta \lambda^3=\lambda^{-2}-\alpha \lambda^{-1}-\beta \lambda^{-3}.$$
But in any case this would yield a non-zero monic annihilating polynomial of degree at most $6$ for $\lambda$, with constant coefficient $-1$.
Yet $p$ is the minimal polynomial of $\lambda$, with degree at least $6$ and $p(0)=1$. This is a contradiction.

We conclude that $h$ can be chosen in $S$ so that $\tr(h \lambda^2) \neq 0$ and $\tr(h \lambda)=\tr(h\lambda^3)=0$, and in that case it is clearly non-zero.

With the specific choice we have just made, $P:=\Vect(1,\lambda^2)$ is an $s$-regular plane and it is $s$-orthogonal to $u(P)=\Vect(\lambda,\lambda^3)$.
\end{proof}

\begin{lemma}\label{lemma:typeItoIVfixplanen>1}
Assume that $\F$ is finite.
Let $n \geq 2$ be an integer. Let $(s,u)$ be an s-pair in which $u$ is cyclic with minimal polynomial
$p^n$ for some $p$ which is either an irreducible monic palindromial or the product $qq^\sharp$ for some $q \in \Irr(\F)$ such that $q^\sharp \neq q$.
Assume further that either $\deg(p)>2$, or $p=t^2+1$ and $n>2$.
Then there exists an $s$-regular plane $P$ such that $u(P) \bot_s P$.
\end{lemma}

\begin{proof}
We will resort again to an algebraic model of $(s,u)$, but with a slightly more complicated form than in the proof of the preceding lemma.
We consider the ring $\F[t,t^{-1}]$ of Laurent polynomials in one variable $t$, and the quotient $\F$-algebra $R:=\F[t,t^{-1}]/(p^n)$.
We also write $p=t^d m(t+t^{-1})$ where $\deg(p)=2d$. Then $m$ has degree $d$ and is irreducible. We write $\lambda$ for the coset of $t$ in $R$, and we set $u' : x \in R \mapsto \lambda x$. Note that $u'$ is $\calC(p^n)$.

This time around, we denote by $\tr$ the trace of the residue field $\L:=\F[t,t^{-1}]/(m(t+t^{-1}))$ over $\F$.
Again, $\L$ is equipped with the involution $x \mapsto x^\bullet$ that takes the coset of $t$ to its inverse, and we note that
the symmetric bilinear form
$$b : (x,y) \in \L^2 \mapsto \tr(xy)$$
is nondegenerate and that $\tr$ is invariant under conjugation.
The $\F$-linear subspaces
$$H:=\{x \in \L : x^\bullet=x\} \quad \text{and} \quad S:=\{x \in \L : x^\bullet=-x\}$$
are $b$-regular for this bilinear form, and are $b$-orthogonal to one another.

Now, we introduce a nice symplectic structure on the $\F$-algebra $R$ for which $u'$ is a symplectic automorphism.
We consider the linear subspace
$$A_0:=\Vect(t^k)_{-d <k <d} \oplus \F (t^d-t^{-d})$$
(which is stable under the canonical involution of $\F[t,t^{-1}]$)
and the graded decomposition
$$\F[t,t^{-1}]=\underset{k \in \N}{\bigoplus} \underbrace{m^k(t+t^{-1}) A_0}_{A_k}\; ,$$
and we denote by $\pi_k$ the projection onto $A_k$ that is associated with this decomposition.
The decomposition is valid as, by using the identity $t^d+t^{-d}=m(t+t^{-1})$ mod $\Vect(t^k)_{-d<k<d}$,
one checks by induction that
$$\forall N \in \N, \quad \underset{k=0}{\overset{N}{\sum}} A_k=\Vect(t^k)_{-(N+1)d < k < (N+1)d} \oplus \F \bigl(t^{(N+1)d}-t^{-(N+1)d}\bigr).$$
Each $A_k$ is stable under the canonical involution, and hence each $\pi_{k}$ commutes with it.
Finally, $\sum_{k \geq N} A_k$ is the ideal generated by $m^N(t+t^{-1})$, for all $N \geq 0$.

The projection onto the residue field $\L:=\F[t,t^{-1}]/(m(t+t^{-1}))$ induces an $\F$-vector space isomorphism $\varphi : A_0 \overset{\simeq}{\rightarrow} \L$, which we will use to identify the elements of $A_0$ with elements of $\L$.

We can now define the relevant symplectic form on $R$. We start by fixing an arbitrary $h_0 \in S \setminus \{0\}$
(to be adjusted later). Then, for $x,y$ in $\F[t,t^{-1}]$,
we extract $\pi_{n-1}(x^\bullet y)$, which is an element of $A_{n-1}$, then map it back to $A_0$ by dividing by $m^{n-1}(t+t^{-1})$,
and finally we compute the trace
$$\tr\Bigl(h_0\,\varphi(((m^{n-1}(t+t^{-1}))^{-1}\pi_{n-1}(x^\bullet y))\Bigr) \in \F.$$
The resulting mapping  is $\F$-bilinear and it vanishes whenever $x$ or $y$ is a multiple of $m(t+t^{-1})^n$
(because then the product $x^\bullet y$ belongs to $\underset{k \geq n}{\oplus} A_k$ and hence its projection on $A_{n-1}$ is zero).
It follows that, by identifying $A_0$ with the residue field $\L$
thanks to $\varphi$, the previous mapping induces an $\F$-bilinear mapping
$s'$ on $R^2$. And because $\lambda^\bullet \lambda=1$ we see that $u'$ is an $s'$-isometry.
It is then classical to check that $s'$ is actually a symplectic form on the $\F$-vector space $R$:
\begin{itemize}
\item First of all it is clear that $s'$ is $\F$-bilinear;
\item Next, $s'$ is alternating because, since
$\pi_{n-1}$ commutes with the involution $x \mapsto x^\bullet$, we find that
$\pi_{n-1}(x^\bullet x)$ is Hermitian, and hence $h_0\,\varphi(((m^{n-1}(t+t^{-1}))^{-1}\pi_{n-1}(x^\bullet x))$ is skew-Hermitian, to the effect that its trace is zero;
\item Finally, let us prove that $s'$ is nondegenerate: let $x \in R \setminus \{0\}$, which we write $x=m(t+t^{-1})^k a^\bullet+m(t+t^{-1})^{k+1} a'$ for some
integer $k \in \lcro 0,n-1\rcro$, some $a \in A_0 \setminus \{0\}$ and some $a' \in \F[t,t^{-1}]$. Let us take an arbitrary $w \in A_0$, and set $y:=m(t+t^{-1})^{n-1-k} w$.
then $s'(x,y)=\tr(h_0 \varphi(\pi_0(aw)))=\tr(h_0 \overline{aw})$, where $\overline{aw}$ stands for the coset of $aw$ modulo $m(t+t^{-1})\F[t,t^{-1}]$ (in $\L$).
Now, if $x$ belonged to the radical of $s'$ we would obtain $b(h_0\overline{a},\overline{w})=\tr(h_0 \overline{a} \overline{w})=0$ for all $w \in A_0$,
thereby contradicting the non-degeneracy of $b$. Hence $s'(x,-) \neq 0$, and we conclude that $s'$ is nondegenerate.
\end{itemize}
Remember now that $u'$ is $\calC(p^n)$. As $\F$ is finite and $p$ has no root in $\{1,-1\}$, there is no Wall invariant to consider,
and we deduce that $(s,u)$ is isometric to $(s',u')$. Hence no generality is lost in assuming that $(s,u)=(s',u')$, an assumption we will now make.

It remains to construct the plane $P$.
To do so, we fix an arbitrary non-zero skew-Hermitian element $h_1 \in \L \setminus \{0\}$, and we take
$x$ as the coset of $1$ in $R$ and $y$ as the coset of $m(t+t^{-1})^{n-1}\varphi^{-1}(h_1)$ in $R$.
Obviously $x,y$ are linearly independent over $\F$ because $n>1$.
We take
$$P:=\Vect(x,y).$$
Then $x^\bullet x=1_R$ and hence $s(x,x)=0$ because $\pi_{n-1}(1_R)=0$,
while $y^\bullet y=0$ (because $n\geq 2$) and hence $s(y,y)=0$.
Finally $s(x,y)=\tr(h_0 h_1)$, and hence $P$ is $s$-regular if and only if $\tr(h_0 h_1) \neq 0$.

Next, $x^\bullet (\lambda x)$ is the coset of $t$ in $R$. Yet $t \in A_0$ if $\deg p>2$;
otherwise $m(t+t^{-1})=t+t^{-1}$, $t=\frac{1}{2}(t-t^{-1})+\frac{1}{2}(t+t^{-1}) \in A_0+A_1$, and the assumptions require that $n>2$.
Hence in any case $\pi_{n-1}(t)=0$, and hence $s(x,\lambda x)=0$.
Moreover $s(y,\lambda y)=0$ because $y^\bullet(\lambda y)$ is the coset of an element of the ideal $\underset{k \geq 2(n-1)}{\sum} A_k$,
with projection zero on $A_{n-1}$.

Finally, we note that $s(y,\lambda x)=-s(\lambda x,y)=-s(x,\lambda^{-1} y)$ and
$s(x,\lambda y)-s(x,\lambda^{-1}y)=s(x,(\lambda-\lambda^{-1})y)$.
Noting that $\lambda-\lambda^{-1}$ projects modulo $m(t+t^{-1})$ to a skew-Hermitian element $h_2$ of $\L$,
we find $s(x,(\lambda-\lambda^{-1})y)=\tr(h_0h_2h_1)=0$ because $h_0h_2h_1$ is skew-Hermitian.
It follows that
$$s(x,\lambda y)=s(x,\lambda^{-1}y)=\frac{1}{2}\,s(x,(\lambda+\lambda^{-1})y).$$
Hence
$$P \bot_s u(P) \; \Leftrightarrow \; s(x,(\lambda+\lambda^{-1})y)=0.$$

Hence, it remains to prove that $h_0$ and $h_1$ can be chosen in $S$ so that
$\tr(h_0h_1) \neq 0$ and $\tr(h_0h_1z)=0$, where $z$ stands for the coset of $t+t^{-1}$ in $\L$.
If $p=t^2+1$ then $z=0$. So, it suffices to choose $h_0$ and $h_1$
so that $\tr(h_0h_1) \neq 0$, which is possible because $S$ is non-zero and $b$-regular.

Assume finally that $\deg p>2$. Then $z$ is invertible in $\L$.
Assume that there is no pair $(h_0,h_1)\in S^2$ such that $\tr(h_0h_1)\neq 0$ and $\tr(h_0h_1z)=0$.
Then, for all $h_0,h_1$ in $S$, $\tr(h_0h_1z)=0 \Rightarrow \tr(h_0h_1)=0$.
Noting that $\psi : h \mapsto z^{-1}h$ is an endomorphism of $S$, and that $S$ is $b$-regular, this yields that
$\psi$ leaves invariant the $\tr$-orthogonal complement in $S$ of every $1$-dimensional subspace, and hence it leaves invariant
every linear hyperplane of $S$; classically this is enough to see that $\psi$ leaves invariant every $1$-dimensional subspace of $S$,
and in particular $z^{-1}(\lambda-\lambda^{-1})=\alpha (\lambda-\lambda^{-1})$ for some $\alpha \in \F$.
But then $\alpha (t+t^{-1})=1$ mod $m(t+t^{-1})$, contradicting the fact that $\deg p>2$.

Hence in any case there exists a pair $(h_0,h_1)\in S^2$ such that $\tr(h_0h_1)\neq 0$ and $\tr(h_0h_1z)=0$,
and with this choice we obtain that $P=\Vect(x,y)$ is an $s$-regular plane and that $P \bot_s u(P)$
(note that the condition $\tr(h_0h_1) \neq 0$ ensures that both $h_0$ and $h_1$ are non-zero).
\end{proof}

\subsection{Tackling cells of type V}\label{section:typeV}

\begin{lemma}\label{lemma:typeVfixplane}
Let $(s,u)$ be an s-pair that is an indecomposable cell of type V and dimension at least $4$. Then there exists an involution $i \in \Sp(s)$
together with an $s$-regular plane $P$ such that $iu$ fixes every vector of $P$.
\end{lemma}

\begin{proof}
We will work backwards.
We start with general results on cells of type V and with eigenvalue $1$.
Let $(V,s)$ be an arbitrary symplectic space of dimension $2n>0$, and
$(e_1,\dots,e_n,f_1,\dots,f_n)$ be an $s$-symplectic basis of $V$.
Let us consider a strictly upper-triangular matrix $N=(n_{i,j}) \in \Mat_n(\F)$ and an arbitrary matrix $B=(b_{i,j}) \in \Mat_n(\F)$.
The automorphism $u$ of $V$ with matrix in the basis $(e_1,\dots,e_n,f_1,\dots,f_n)$ equal to
$$\begin{bmatrix}
I_n+N & B \\
0 & (I_n+N)^\sharp
\end{bmatrix}$$
is $s$-symplectic if and only if $(I_n+N)^{-1}B$ is symmetric. Assuming that such is the case,
it is easily checked that $u$ has characteristic polynomial $(t-1)^{2n}$, and it is $\calC((t-1)^{2n})$ if and only if $b_{n,n} \neq 0$ and $I_n+N$ is cyclic (i.e.\ $n_{i,i+1} \neq 0$ for all $i \in \lcro 1,n-1\rcro$).
In that case one checks that
$$(u-\id)^{2n-1}(f_1)=(-1)^{n-1} b_{n,n} \prod_{k=1}^{n-1}(n_{k,k+1})^2\,e_1.$$
Since $(u+\id)u^{-(n-1)}(e_1)=2e_1$ we obtain
$$s\left(f_1,(u-u^{-1})(u+u^{-1}-2\id)^{n-1}(f_1)\right)=2 (-1)^{n} b_{n,n} \left(\prod_{k=1}^{n-1} n_{k,k+1}\right)^2,$$
and hence the quadratic form attached to the Wall invariant $(s,u)_{t-1,2n}$ takes the value $2(-1)^{n} b_{n,n}$.

Now, let us go back to an arbitrary indecomposable s-pair $(s,u)$ of type V and dimension $2n \geq 4$. Without loss of generality,
we can assume that the corresponding eigenvalue is $1$, and we can choose $\alpha \in \F \setminus \{0\}$
in the range of the quadratic form that is attached to the Wall invariant $(s,u)_{t-1,2n}$.
We denote by $V$ the underlying vector space of $(s,u)$.
We choose an arbitrary $s$-regular plane $P$ of $V$ together with a symplectic basis $(e_2,\dots,e_n,f_2,\dots,f_n)$ of $P^{\bot_s}$,
as well as a symplectic basis $(e_1,f_1)$ of $P$. We take the upper-triangular Jordan matrix $N:=(\delta_{i+1,j})_{1 \leq i,j \leq n-1} \in \Mat_{n-1}(\F)$
and the symmetric matrix $C \in \Mats_{n-1}(\F)$ with exactly one non-zero entry, which equals $1$ and is located at the $(n-1,n-1)$-spot.
Finally, we take an arbitrary scalar $\alpha \in \F \setminus \{0\}$. Then we take $B:=\frac{(-1)^{n}}{2}\alpha (I_{n-1}+N)C$, and we note that
$b_{n-1,n-1}=\frac{(-1)^{n}}{2}\,\alpha$.
We define $u'$ as the symplectic automorphism of $P^{\bot_s}$ with matrix
$$\begin{bmatrix}
I_{n-1}+N & B \\
0 & (I_{n-1}+N)^\sharp
\end{bmatrix}$$
in the basis $(e_2,\dots,e_n,f_2,\dots,f_n)$, and we extend it to a symplectic automorphism of $V$ by taking
$u'(x)=-x$ for all $x \in P$.
Finally, we take
$$D:=\begin{bmatrix}
-1 & L  \\
0 & I_{n-1}
\end{bmatrix} \in \Mat_n(\F), \quad
\text{where} \; L:=\begin{bmatrix}
1 & 0 & \cdots & 0
\end{bmatrix},$$
we note that $D^2=I_n$, and we define $i$ as the linear involution of $V$ with matrix
$D \oplus D^\sharp$ in the basis $(e_1,\dots,e_n,f_1,\dots,f_n)$. Note that $i \in \Sp(s)$.
Then the matrix of $iu'$ in that basis reads
$$\begin{bmatrix}
I_n+N' & B' \\
0 & (I_n+N')^\sharp
\end{bmatrix}$$
where $b'_{n,n}=b_{n-1,n-1}$ and $N'$ is upper-triangular with $n'_{k,k+1} \neq 0$ for all $k \in \lcro 1,n-1\rcro$.
It follows from the early considerations in this proof that $iu'$ is an indecomposable cell of type V with eigenvalue $1$
and that $\alpha$ is in the range of the quadratic form that is attached to $(s,iu')_{t-1,2n}$.
Hence $(s,iu')_{t-1,2n} \simeq (s,u)_{t-1,2n}$, and we deduce that
$iu'$ is conjugated to $u$ in $\Sp(s)$. Hence $u$ is i-adjacent to a conjugate of $-u'$ in $\Sp(s)$, which yields the claimed result.
\end{proof}

\subsection{Completing the proof}\label{section:final}

We are almost ready to complete the proof of point (b) of Theorem \ref{theo:main}.
So, we let $(s,u)$ be an s-pair with dimension $n \geq 6$ such that $n=2$ mod $4$,
and we seek to prove that $u$ is $5$-reflectional in $\Sp(s)$.
We proceed by induction on $n$.
If $n=6$, this is already known from Section \ref{section:dim6}, so now we assume that $n \geq 10$.

If $u$ has a Jordan cell of odd size for an eigenvalue in $\{\pm 1\}$, then
we can find a decomposition $u=u_1 \botoplus u_2$ where $u_2$ is an indecomposable cell of type VI,
and then the rank of $u_1$ is a multiple of $4$. Then $u_2$ is $2$-reflectional in the corresponding symplectic group, whereas by points (a) and
(c) of Theorem \ref{theo:main} (or trivially if $u_1$ has rank $0$) we see that $u_1$ is $5$-reflectional in the corresponding symplectic group.
Hence $u$ is $5$-reflectional in $\Sp(s)$.

Assume now that $u$ has no Jordan cell of odd size for an eigenvalue in $\{\pm 1\}$.
Let us then consider a decomposition $u=u_1 \botoplus \cdots \botoplus u_r$ into indecomposable cells, by non-increasing order
of rank.
\begin{itemize}
\item If $r \geq 2$ and $\rk u_{r-1}=2$ then we regroup $u=(u_1\botoplus \cdots \botoplus u_{r-2}) \oplus (u_{r-1} \botoplus u_r)$,
and we note that $u_{r-1} \botoplus u_r$ is $5$-reflectional in the corresponding symplectic group, whereas by induction so is
$u_1\botoplus \cdots \botoplus u_{r-2}$ (because it has rank $n-4$).
\item If $r \geq 3$ and $\rk u_{r-1}>2$, or if $r=2$ and $\rk u_r>2$, then we regroup $u=u_1 \botoplus (u_2\botoplus \cdots \botoplus u_r)$, and again
we obtain the result by induction and by using points (a) and (c) of Theorem \ref{theo:main}.
\end{itemize}
Hence, only two cases remain.

\noindent
\textbf{Case 1:} $(s,u)$ is an indecomposable cell of type I to V, with dimension at least $10$.\\
\textbf{Case 2:} $(s,u)$ is the orthogonal direct sum of an indecomposable cell of type I to V (and dimension at least 8)
with an indecomposable cell of dimension 2.

\vskip 3mm
Let us start with the special case of $\F_3$.
The monic polynomials with degree $2$ and constant coefficient $1$ over $\F_3$ are $(t-1)^2$, $(t+1)^2$ and $t^2+1$.
Hence, by combining Lemmas \ref{lemma:typeItoIVfixplanen=1}, \ref{lemma:typeItoIVfixplanen>1} and
\ref{lemma:typeVfixplane}, we obtain that for every indecomposable cell $(s',u')$ of type I to V and dimension at least $8$,
there exists an involution $i \in \Sp(s')$ such that $iu'$ fixes every vector of a regular $s'$-plane.
Hence, in any of the above two cases, there exists an involution $i \in \Sp(s)$
such that $iu$ fixes every vector of some regular $s$-plane $P$. Then, as explained at the beginning of Section \ref{section:dim2mod4},
applying point (a) of Theorem \ref{theo:main} to $(iu)_{P^\bot}$ yields that $iu$ is $4$-reflectional in $\Sp(s)$,
and we conclude that $u$ is $5$-reflectional in $\Sp(s)$.

The same method works if $\F$ is finite and $|\F|>3$, with the exception of the following two special cases, where
we set $m:=\frac{n-2}{4}\cdot$

\noindent \textbf{Special case 1:} $(s,u)$ is an indecomposable cell of type II or IV, with minimal polynomial $p^{2m+1}$
for some monic $p \in \F[t]$ with degree $2$, no root in $\{\pm 1\}$ and $p(0)=1$.

\noindent \textbf{Special case 2:} $(s,u)$ is the orthogonal direct sum of two indecomposable cells,
one of which is of type I or III with minimal polynomial $p^{2m}$
for some monic $p \in \F[t]$ with degree $2$, no root in $\{\pm 1\}$ and $p(0)=1$, and the other of which has dimension $2$.

In both cases, we can find a totally $s$-singular subspace $W$ of $V$ that is stable under $u$,
has dimension $2(m-1)$, and such that $u_{W}$ is cyclic with minimal polynomial $q$ such that $q(0)=1$. Then, the result is obtained thanks to our last lemma:

\begin{lemma}\label{lemma:last}
Assume that $\F$ is finite with $|\F|>3$.
Let $n=4m+2 \geq 10$ where $m$ is an integer. Let $(s,u)$ be an s-pair. Assume that there exists
a totally $s$-singular subspace $W$ that is stable under $u$,
has dimension $2(m-1)$, and is such that $u_W$ is cyclic with minimal polynomial
$q$ such that $q(0)=1$. Then $u$ is $5$-reflectional in $\Sp(s)$.
\end{lemma}

\begin{proof}
Denote by $V$ the underlying vector space of $(s,u)$.
Let us take a complementary subspace $W'$ of $W^\bot$ in $V$. Let us take a symplectic basis
$(e_1,\dots,e_{2m-2},f_1,\dots,f_{2m-2})$ of $W\oplus W'$ in which
$(e_1,\dots,e_{2m-2})$ is a basis of $W$ and
$(f_1,\dots,f_{2m-2})$ is a basis of $W'$. Let us take a symplectic basis $(g_1,\dots,g_6)$ of $(W \oplus W')^\bot$.
Then $\bfB:=(e_1,\dots,e_{2m-2},g_1,\dots,g_6,f_1,\dots,f_{2m-2})$ is a basis of $V$ and the matrix of $u$ in that basis
reads
$$\begin{bmatrix}
A & ? & ? \\
0 & B & ? \\
0 & 0 & A^\sharp
\end{bmatrix}$$
where $A \in \GL_{2m-2}(\F)$ is cyclic with minimal polynomial $q$,
and $B \in \Sp_6(\F)$.
As seen in Section \ref{section:dim6}, we can find an involutory matrix $B_1 \in \Sp_6(\F)$
such that $B_1B$ is $4$-reflectional in $\Sp_6(\F)$.

By Lemma \ref{lemma:polynomp(0)-1}, we can find an $r \in \Irr(\F)$ with degree $2m-2$
such that $r(0)=-1$.
Next, by Lemma \ref{lemma:cyclicfitplain}, we can find an involutory $A_1 \in \GL_{2m-2}(\F)$ such that
$A_1A$ is cyclic with minimal polynomial $r$.
Denote by $i \in \Sp(s)$ the symplectic involution whose matrix in $\bfB$ equals
$A_1 \oplus B_1 \oplus A_1^\sharp$. Then the matrix of $iu$ in the same basis has the form
$$\begin{bmatrix}
A_1A & ? & ? \\
0 & B_1B & ? \\
0 & 0 & (A_1A)^\sharp
\end{bmatrix}.$$
Note that $r$ is not a palindromial (being of even degree and irreducible with $r(0) \neq 1$).
Assume further that $r$ is relatively prime with the characteristic polynomial of $B_1B$. Then by Lemma \ref{lemma:recog}
we find a decomposition $iu =u'_1 \botoplus u'_2$ where $u'_1$ has rank $4m-4$ and $u'_2$
is represented by $B_1B$ in some symplectic basis. Then $u'_2$ is $4$-reflectional in the corresponding symplectic group,
whereas $u'_1$ is $4$-reflectional by point (a) of Theorem \ref{theo:main}.

It remains to prove that a correct choice of $r$ can be made. If $2m-2 \geq 4$, this is straightforward because
if $r$ were not relatively prime with the characteristic polynomial $\chi_{B_1B}$ then the latter (which is a palindromial) would be divided by $rr^\sharp$, which has degree at least $8$.

It remains to consider the case where $m=2$. In that one, we slightly change the choice of $r$: indeed it is not necessary to have an irreducible
polynomial (nor one that is relatively prime with $r^\sharp$), rather we just need a polynomial such that $rr^\sharp$ is relatively prime with $\chi_{B_1B}$
(and, of course, $r$ is monic with degree $2$ and $r(0)=-1$).

Now suppose that our initial choice of $r$ fails. Then as $rr^\sharp$ divides $\chi_{B_1B}$ and both polynomials have constant coefficient $1$,
$\chi_{B_1B}$ has at most two roots in $\F \setminus \{0\}$, counted with multiplicities, and for each such root $\alpha$
the roots of $\chi_{B_1B}$ are actually $\alpha$ and $\alpha^{-1}$, counted with multiplicities.
Then, either $\chi_{B_1B}$ has no root in $\{1,-1\}$, in which case we replace $r$ with $t^2-1$ to get the desired result,
or all the roots of $\chi_{B_1B}$ in $\F \setminus \{0\}$ belong to $\{1,-1\}$, in which case we replace
$r$ with $(t-\alpha)(t+\alpha^{-1})$, where $\alpha$ is chosen in $\F \setminus \{0,1,-1\}$. In any case, we see
that $rr^\sharp$ is relatively prime with $\chi_{B_1B}$, which yields the desired conclusion.
\end{proof}

The proof of point (b) of Theorem \ref{theo:main} is now completed.

\end{document}